\documentclass[letterpaper, 10 pt, conference]{ieeeconf}  
                                                          
\IEEEoverridecommandlockouts                              

\usepackage{graphics} 
\usepackage{epsfig} 
\usepackage{mathptmx} 
\usepackage{amsmath} 
\usepackage{amssymb}  
\usepackage{flushend}

\usepackage{ifthen}
\newboolean{arxiv}

\usepackage{amsthm}  
\usepackage{mathtools}
\usepackage{commath}
\usepackage{amsfonts}
\usepackage{subcaption}
\usepackage[linesnumbered,ruled,vlined]{algorithm2e} 
\SetKwInput{KwInput}{Input}                
\SetKwInput{KwOutput}{Output}
\SetKwInput{KwInitialize}{Initialize} 

\usepackage{ifthen}
\newboolean{showcomments}
\setboolean{showcomments}{true}
\usepackage[usenames,dvipsnames]{color}
\usepackage{todonotes}

\definecolor{bleudefrance}{rgb}{0.19, 0.55, 0.91}
\definecolor{ao(english)}{rgb}{0.0, 0.5, 0.0}

\newcommand{\addcite}[0]{\ifthenelse{\boolean{showcomments}}
{\textcolor{purple}{(add cite(s)) }}{}}%

\newcommand{\pc}[1]{  \ifthenelse{\boolean{showcomments}}
{\todo[inline,color=bleudefrance]{Pengcheng: #1}}{}}
\newcommand{\pcmargin}[1]{\ifthenelse{\boolean{showcomments}}{\marginpar{\color{bleudefrance}\tiny Pengcheng: #1}}{}}

\newcommand{\enrique}[1]{  \ifthenelse{\boolean{showcomments}}
{\todo[inline,color=ao(english)]{Enrique: #1}}{}}
\newcommand{\emmargin}[1]{\ifthenelse{\boolean{showcomments}}{\marginpar{\color{ao(english)}\tiny Enrique: #1}}{}}

\newboolean{showedits}
\setboolean{showedits}{false}
\usepackage[markup=underlined]{changes}
\definechangesauthor[color=bleudefrance]{EM}
\newcommand{\aem}[1]{
\ifthenelse{\boolean{showedits}}
{\added[id=EM]{#1}}
{\!#1\hspace{-4.75pt}}
}
\newcommand{\repem}[2]{
\ifthenelse{\boolean{showedits}}
{\replaced[id=EM]{#1}{#2}}
{\!#1\hspace{-4.75pt}}
}
\newcommand{\dem}[1]{
\ifthenelse{\boolean{showedits}}
{\deleted[id=EM]{#1}}
{}
}
\allowdisplaybreaks
\DeclareMathOperator{\sgn}{sgn}
\DeclareMathOperator{\diff}{diff}

\usepackage{xcolor} 
\usepackage{ifthen}
\usepackage{soul}
\setboolean{showcomments}{true}

\newcommand{\dennice}[1]{\ifthenelse{\boolean{showcomments}}
{\textcolor{purple}{Dennice says: #1}}{}}

\newtheorem{theorem}{Theorem}
\newtheorem{lemma}{Lemma}
\newtheorem{definition}[theorem]{Definition}

\newtheorem{proposition}{Proposition}

\newtheorem{remark}{Remark}

\title{\LARGE \bf
Storage Degradation Aware Economic Dispatch
}

\author{Rajni Kant Bansal, Pengcheng You, Dennice F. Gayme, and Enrique Mallada
\thanks{R. K. Bansal, P. You, D. F. Gayme, and E. Mallada are with the Whiting School of Engineering, Johns Hopkins University, Baltimore, MD 21218, US
        {\tt\small \{rbansal3,pcyou,dennice,mallada\}@jhu.edu}}%

\thanks{This work was supported by NSF through grants ECCS 1711188, CAREER ECCS 1752362 and the US DOE EERE award DE-EE0008006.}
}

\usepackage[noadjust]{cite}
\begin{document}

\setboolean{arxiv}{true}

\maketitle
\thispagestyle{empty}
\pagestyle{empty}

\begin{abstract}

In this paper, we formulate a cycling cost aware economic dispatch problem that co-optimizes generation and storage dispatch while taking into account cycle based storage degradation cost. Our approach exploits the Rainflow cycle counting algorithm to quantify storage degradation for each charging and discharging half-cycle based on its depth. We show that the dispatch is optimal for individual participants in the sense that it maximizes the profit of generators and storage units, under price taking assumptions. We further provide a condition under which the optimal storage response is unique for given market clearing prices. Simulations using data from the New York Independent System Operator (NYISO) illustrate the optimization framework. In particular, they show that the generation-centric dispatch that does not account for storage degradation is insufficient to guarantee storage profitability. 
\end{abstract}

\section{Introduction}
The power system is undergoing rapid changes due to increased penetration of renewable energy sources and the desire for a reduced carbon footprint. However, the intermittency of popular renewable sources, e.g., solar and wind energy, coupled with new variations in load patterns due to demand-side management and devices such as electric vehicles, are affecting system reliability \cite{center2020annual,ipakchi2009grid,Magin2018grid}.
Energy storage systems (ESS) have been widely proposed as means to provide the grid services required to maintain grid reliability and power quality 
\cite{denholm2019potential,vazquez2010energy,caiso2019paper,ribeiro2001energy}.

Lithium-ion based battery storage is one of the fastest growing storage modalities for the power grid.
~\cite{storage2018trend}. However, in contrast with traditional generators, the cost of dispatching storage cannot be directly quantified in terms of supplied power alone. For example, degradation due to numerous charging and discharging half-cycles plays an important role in the operational cost associated with battery storage~\cite{he2015optimal,song2019multi}. However, these and other storage specific costs are not currently accounted for in market settlements and negatively affect the profitability of storage~\cite{d2016new}.

A number of storage degradation models that enable equipment owners to account for storage degradation in their cost/benefit analysis have been proposed~\cite{abdulla2016optimal,rosewater2019battery}. The most widely used are cycle-based degradation models that quantify the cost of each half-cycle (i.e. charging or discharging) based on its depth, defined as the ratio of the energy charged (or discharged) to the capacity of the storage. These models are generally combined with the \emph{Rainflow} cycle counting algorithm, which extracts charging and discharging half-cycles from a storage State of Charge (SoC) profile.
Several approaches based on the Rainflow algorithm have been proposed to incorporate storage degradation into storage operators market participation strategies~\cite{he2015optimal, niu2016sizing, ke2014control, shi2017optimal}. Although these solutions account for the intrinsic degradation of storage actions, they require mapping SoC profiles into actionable energy buy/sell decisions for each time slot, which makes any guarantee of efficiency or optimality, at best, an approximate statement~\cite{xu2017optimal}.

This paper seeks a different approach. Instead of mapping storage degradation into a sequence of buy and sell energy transactions to be submitted to the market, we argue that it is better to formulate a market that intrinsically allows for storage units to participate in it. To this end, we formulate an economic dispatch problem, that takes into account the cost of a storage in terms of the degradation it incurs. For concrete insights, our formulation simply considers one generator unit and one storage unit. However, our results do not critically depend on this assumption.
Using this formulation, and leveraging recent results on convexity of cycle-based degradation cost~\cite{shi2017optimal}, we show that, not only it is possible to efficiently find a storage and generation scheduling that minimizes the overall operational cost, but also that such scheduling is optimal from the viewpoint of individual participants. More precisely, by means of dual decomposition, we show that, under price taking assumptions, the obtained allocation is incentive compatible, i.e., it simultaneously maximizes individual profit of both the generator and the storage. 

\textcolor{black}{ Our work also provides} a novel formulation of the Rainflow algorithm that analytically represents its input-output relation as a piece-wise linear map from the SoC profile to half-cycle depths. This mapping further allows a representation in terms of graph incidence matrices where nodes represent time slots and edges describe charge or discharge half-cycles. Using this reformulation, we provide conditions on the incidence matrix 
under which the optimal storage response is unique for given market clearing prices. Numerical simulations illustrate the importance of accounting for storage cycling cost in the economic dispatch problem by comparing the total cost of operation for different storage sizes and storage capital costs obtained under the traditional (generation centric) dispatch to those of the proposed storage degradation aware dispatch.

\textcolor{black}{The rest of the paper is organized as follows. In Section \ref{Section II}, we define the economic dispatch problem and formulate the optimal response problems for individual participants given market clearing prices. In Section \ref{Section III}, we introduce the storage cost model.} The structural results of the economic dispatch problem that feature the optimality of dual decomposition and the uniqueness of storage response to market clearing prices 
are presented in Section \ref{Section IV}. Numerical analysis and conclusions are provided in Sections~\ref{Section V} and \ref{Section VI}, respectively.
\subsubsection*{Notation} Given a set $S$ of \emph{time indices} $S\subseteq \{0,\dots,T\}$, $S[j]$ denotes the $j^{\mbox{th}}$ smallest element in $S$, e.g., $S=\{3,1,5\}$, $S[2]=3$. Given a vector $x=(x_0,\dots,x_T) \in\mathbb{R}^{T+1}$ and the set $S\subseteq\{0,\dots,T\}$, $x_S\in\mathbb{R}^{|S|}$ denotes the order preserving vector of elements indexed by $S$ (preserving $x$ order). We define the element $x_S(j):=x_{S[j]}$, for example, for the same set $S=\{3,1,5\}$, $x_S(2)=x_{S[2]}=x_3$. For a vector $x_S$ and an index 
 $j \in \{2,3,\cdots,|S|-2\}$, we define the triple difference operation 
\[
(\Delta_{j-1},\Delta_{j},\Delta_{j+1}) := \diff(x,S,j)
\]
with $\Delta_{j}:=|x_S(j+1)-x_S(j)|$, which will be used to identify cycles. 
We next define a direction operation pointing from $t_1$ to $t_2$, which will be used in adding (directed) edges to a directed graph (digraph). For a vector $x_S$ and an index $j \in \{1,2,\cdots,|S|-1\}$ as
\[
\textcolor{black}{(t,t')} := \textrm{dir}(x,S,j)=\left\{\begin{array}{l}
\left(S[{j+1}],S[j]\right), \mbox{ if } x_S(j+1) \ge x_S(j) \\
\left(S[j],S[{j+1}]\right), \mbox{ otherwise}.
\end{array}\right.
\] 
For a set $S'$ of ordered \emph{time index pairs} $S' \subseteq \{0,\dots,T\} \times \{0,\dots,T\}$, \textcolor{black}{$(t,t') \in S'$} denotes an ordered pair from \textcolor{black}{$t$ to $t'$}. 

\section{\textcolor{black}{Problem Formulation}}  \label{Section II}

\subsection{\textcolor{black}{Economic Dispatch}}
We consider 
a simple case with one generator, one storage element and inelastic demand to gain analytical insights, but all of the results can be generalized.
Suppose the generator is able to output $g_t$ amount of power at time $t$ subject to capacity constraints
\begin{equation}\label{gen_limit_ineq}
    \underline{g} \le g_t \le \overline{g}, \quad t \in \{1,2,\dots, T\},
\end{equation}
where $\underline{g}$ and $\overline{g}$ denote the minimum and maximum generation limits, respectively. We use $g
\in\mathbb{R}^T $ to denote the generation profile. Similarly the demand profile is defined as $D
\in\mathbb{R}^T$, where $D_t$ denotes the inelastic demand at time $t \in \{1,2,...,T\}$.

\textcolor{black}{To account for temporally interdependent storage operation, a multi-slot storage degradation aware 
economic dispatch problem (SDAD) is formulated. The SDAD problem minimizes the total cost of the generator and storage over the horizon while satisfying the given demand profile $D$, subject to their respective operational constraints:}
\begin{subequations}
\label{optm_problem}%
\begin{eqnarray}
    \min_{g,u,x} && \alpha_g g^Tg + \beta_g \mathbf{1}^T g + C_s(x)
    \label{optm_obj} \\
    \text{s.t. } 
    && 
    \eqref{gen_limit_ineq}, \ (u,x) \in \mathcal{S} \nonumber \\
    && D + u = g \label{power_bal_const} 
\end{eqnarray}
\end{subequations}
where the first two terms in the objective function \eqref{optm_obj} represent a standard quadratic cost function for the generator with constant coefficients $\alpha_g>0$ and $\beta_g>0$. \textcolor{black}{$C_S(x)$ represents the cycling cost of storage and the set $\mathcal{S}$ represents the operational constraint of the storage unit, defined explicitly below.} Equation
\eqref{power_bal_const} enforces the power balance. 



\subsection{Storage Operation Model}

\textcolor{black}{ We next characterize the explicit storage operation constraints that define the set $\mathcal{S}$.}
Consider a storage element of capacity $E$, for which the amount of energy stored over a time horizon $\{0,1,\hdots,T\}$ is described by a SoC profile $x \in \mathbb{R}^{T+1}
$ with the initial SoC $x_0=x_o$. The SoC at each time $t$ ($x_t$) is normalized with respect to $E$ such that 
\begin{equation}
    0\leq x_t \leq 1 ,\quad t\in \{0,1,\hdots,T\} \label{Soc_limit_ineq}.
\end{equation}
The charging and discharging rates are denoted by $u\in \mathbb{R}^T 
$ where $u_t > 0 $ (resp. $u_t<0$) represents charging (resp. discharging) at time $t \in \{1,..,T\}$.
We assume the rate $u$ is bounded by the power rating of the device(s)
\begin{equation}
    \underline{u} \leq u_t \leq \overline{u}, \quad t\in \{1,2,\hdots,T\}. \label{SoC_control_limit}
\end{equation}
The SoC evolves according to  
\begin{eqnarray}
    x_{t} =  x_{t-1} + \frac{1}{E}u_{t},\quad t \in \{1,2,\dots,T\}, \nonumber 
\end{eqnarray}
which can be rewritten as
\begin{equation}
    Ax = \frac{1}{E}u ,\label{SoC_evolution}
\end{equation}
where $A \in \mathbb{R}^{T \times ({T+1})}$ is the lower triangular matrix
\begin{equation}
    A =  \begin{pmatrix}
-1 & 1 & 0 &\hdots &0\\
 0 &-1 & 1 &\hdots &0\\
\vdots & \ddots &\ddots & \ddots & \vdots\\
0 & \hdots & 0 & -1 & 1\\
\end{pmatrix}. \label{A_matrix} \nonumber 
\end{equation}
For ease of exposition, we have assumed that the storage operation is lossless and the charging or discharging efficiency is 1. We leave extension to the more general case as a direction for future work. Without loss of generality, we further require that the storage returns to its initial SoC after a complete operation over the time horizon, i.e.,
\begin{equation}
    x^Te_{1} =  x^Te_{T+1}  = x_o , \label{SoC_terminal_const}
\end{equation}
where $e_1:= [1 \ 0 \cdots 0]^T$ and $e_{T+1}:=[0 \cdots 0 \ 1]^T$ are the standard basis vectors in $\mathbb{R}^{T+1}$. \textcolor{black}{The operational constraint set is then given by $\mathcal{S}= \{(u,x) : \eqref{Soc_limit_ineq},\,\eqref{SoC_control_limit},\,\eqref{SoC_evolution},\,\eqref{SoC_terminal_const} \}$}.

\subsection{\textcolor{black}{Individual Subproblems}}

\textcolor{black}{In addition to the optimal dispatch $(g^*,u^*,x^*)$ obtained through the solution of \eqref{optm_problem}, we are also interested in the incentive compatibility of individual participants, i.e., the willingness of the generator and the storage to participate in the dispatch problem. 
Formally, we are interested in finding a set of prices $p \in \mathbb{R}^T$ such that the optimal schedule found by \eqref{optm_problem} is also optimal with respect to the following participant problems, which respectively maximize their profits.}

\noindent \emph{Generator Subproblem}:
\begin{subequations} \label{subproblem1}    
\begin{eqnarray}
\max_{g} &&  p^T g - \alpha_g g^Tg - \beta_g \mathbf{1}^T g  \label{subproblem1_obj}  \\
\mathrm{s.t.} && \eqref{gen_limit_ineq}
\end{eqnarray}
\end{subequations}
\noindent
\emph{Storage Subproblem}:
\begin{subequations}\label{subproblem2}
\begin{eqnarray}
     \max_{u,x }  &&  - p^T u -  C_s(x)   \label{subproblem2_obj} \\
     \mathrm{s.t.} && (u,x) \in \mathcal{S}
\end{eqnarray}
\end{subequations}
\textcolor{black}{where the first term in the objective function \eqref{subproblem1_obj} and \eqref{subproblem2_obj} represent the revenue from the market.}  

\textcolor{black}{The additional requirement of ensuring that the solution of the SDAD is also optimal with respect to providing solutions to \eqref{subproblem1} and \eqref{subproblem2}, makes the problem of finding this prices challenging. The next section describes a means to obtain an analytical expression for storage cycling cost which will then be exploited to obtain the incentive compatible economic dispatch.}







\section{Storage Cost Model} \label{Section III}

\subsection{Rainflow Algorithm Based Cycling Cost}

The operational cost of a storage is primarily associated with the degradation incurred during each cycle. 
Here we adopt the 
model \cite{shi2017optimal} that uses a cycle stress function $\Phi(\cdot): [0,1] \mapsto [0,1]$ to quantify the normalized capacity degradation incurred by each charging or discharging half-cycle as a function of the cycle depth. A full cycle that consists of both charging and discharging half-cycles of the same depth $d_i$ then incurs a degradation of $2\Phi(d_i)$. We adopt the empirically convex stress function $\Phi(d_i)$ of the form 
$$ \Phi(d_i) := \frac{\alpha_b}{2} d_i^{\beta_b}$$ with coefficients $\alpha_b > 0$ and $\beta_b>1$~\cite{shi2017optimal}. 

Given a vector $d
\in\mathbb{R}^T$ that summarizes the depths of all half-cycles\footnote{ In order to maintain a fixed-size vector $d$ we fill in zero-depth half-cycles at the tail when there are less than $T$ half-cycles.} in the time horizon $\{0,1,\dots,T\}$, the total cycling cost of storage is given by 
\begin{eqnarray}
    C_s(d) = BE \Big(\sum_{i=1}^{T}\Phi(d_i)\Big) ,
    \label{str_cost_original}
\end{eqnarray}
where $B$ is the unit capital cost per kilowatt-hour of storage capacity. $BE$ therefore amounts to the storage replacement cost. $C_s(\cdot):\mathbb{R}^T\mapsto \mathbb{R}$ denotes the cycling cost function in terms of half-cycle depths.

In order to map the SoC profile $x$ to cycle depths $d$ we introduce 
a 
cycle identification approach (Algorithm~\ref{alg:rainflow}) based on the Rainflow algorithm \cite{lee2011rainflow}. In addition to the vector of half-cycle depths $d$, our algorithm outputs a set $S_f$ of ordered \emph{time index pairs}, which are used to compute the cycle depth of full-cycles from $x$, and a residual set $S_r$ of \emph {time indices}, which are used to compute residual individual half-cycle depths from $x$. While the sets $S_f$ and $S_r$ are not standard outputs of the Rainflow algorithm, they will be particularly useful in our reformulation of the Rainflow algorithm as a piece-wise affine map.


The main stages of our version of the Rainflow algorithm follow:
\begin{itemize}
    \item (\emph{Switching Time Identification}):
    Starting with $S_r=\{0,T\}$, traverse $x$ from $x_0$ to $x_T$ and store in $S_r$ the time indices where the profile $x$ changes direction, e.g., switches between charging and discharging. 
    This procedure comprises steps~\ref{sgn_step_start}-\ref{sgn_step_end} in Algorithm~\ref{alg:rainflow}. 
    \item (\emph{Full Cycle Extraction}): Looping through $j = 2 : |S_r|-2$, compute the net SoC changes between four consecutive switching points, i.e.,
    $(\Delta_{j-1},\Delta_{j},\Delta_{j+1}) := \diff(x,S_r,j)$.
    If $\Delta_{j-1} \geq \Delta_j$ and $\Delta_{j+1} \geq \Delta_j$, extract a full cycle of depth $\Delta_{j}$ i.e., remove $S_r[j]$ and $S_r[j+1]$ from $S_r$ and add $\textrm{dir}(x,S_r,j)$ to $S_f$. 
    The extracted charging and discharging half-cycle
     of depth $\Delta_j$ is added into the cycle depth vector $d$. This stage is described in steps \ref{check}-\ref{check2} of Algorithm \ref{alg:rainflow}.
    \item (\emph{Half-cycle Extraction)}:
    Once all full cycles are extracted, iterate through $j = 1:|S_r|-1$ to add the depths of all remaining half-cycles. This stage is described in steps~\ref{check3}-\ref{end} of Algorithm~\ref{alg:rainflow}.  
\end{itemize}
We illustrate this procedure using an example SoC profile shown in Fig.~\ref{fig:rainflow_example1}. After steps~\ref{sgn_step_start}-\ref{sgn_step_end} we start with sets $S_r = \{0,1,2,3\}$ and $S_f=\emptyset$. Since $\Delta_1 \geq \Delta_2$ and $\Delta_3 \geq \Delta_2$, with $(\Delta_{1},\Delta_{2},\Delta_{3}) := \diff(x,S_r,2)$, a full cycle of depth $x_1-x_2$ is extracted (see the center panel of Fig.~\ref{fig:rainflow_example1}). This operation leaves the residual charging half-cycle from $x_0$ to $x_3$, shown in the right panel of Fig.~\ref{fig:rainflow_example1}. The output of the algorithm is then $d=[x_1-x_2, x_1-x_2, x_3-x_0]^T$, $S_r = \{0,3\}$, and $S_f=\{(1,2)\}$.

\begin{figure}[htp]
    \centering
    \includegraphics[width=8cm]{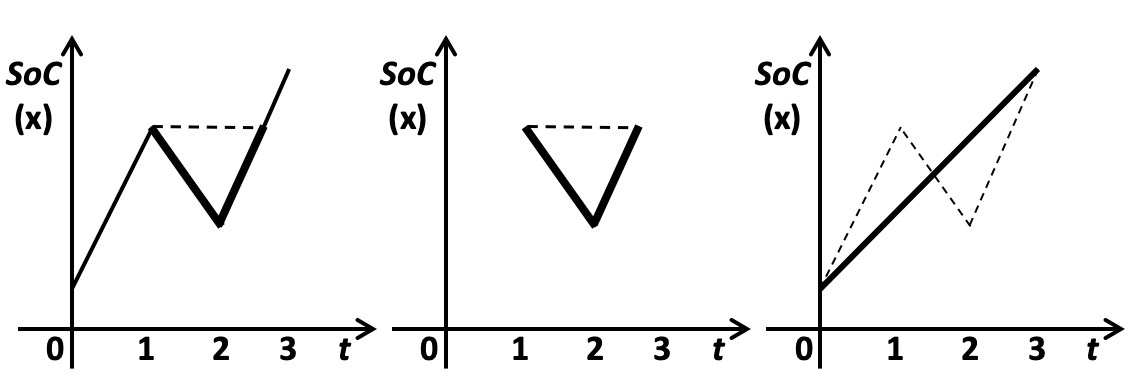}
    \caption{An example of an SoC profile, its extracted full cycle, and the residual half-cycle from left to right.}%
    \label{fig:rainflow_example1}%
\end{figure}%
\vspace*{-.7cm}
\begin{algorithm}[ht]
\SetAlgoLined
\KwResult{vector $d  \in \mathbb{R}^T$, sets $S_f$ and  $S_r$}
\KwInput{SoC Profile $x$}
\KwInitialize{$S_r = \{0,T\}$, $S_{f}=\emptyset$, $d = \Vec{0}$, counter $k=1$} 
\caption{Rainflow Cycle Counting} \label{alg:rainflow}
\For{$t = 1:T-1$ \label{sgn_step_start}}{
    \If{$\sgn(x_{t+1}-x_t) == - \sgn(x_t - x_{t-1})$}{
        $S_r = S_r + \{t\}$\;}}\label{sgn_step_end}
\emph{Cycle} = $\mathrm{true}$\;
\While{Cycle == $\mathrm{true}$ \text{\rm\bf and } $|S_r| > 3$ \label{check}}{
    \emph{Cycle} = \textrm{false}\;
    {\For{$j = 2:|S_r|-2$}{
    $(\Delta_{j-1},\Delta_{j},\Delta_{j+1})=\diff(x,S_r,j)$\;
    \If {$\Delta_{j-1} \geq \Delta_{j}$, and $\Delta_{j+1} \geq \Delta_{j}$} {$S_r\!=\! S_r\!-\!\{S_r[j],S_r[j\!+\!1]\}$,
    $S_{f} \!=\! S_{f}\!+\textrm{dir}(x,S_r,j)$  
    \;
    $d(k) = d(k+1) = \Delta_j$ and $k=k+2$\;
    \emph{Cycle} = \textrm{true}, and restart from step 
\ref{check}\;}}}}\label{check2}
 \For{j=1:$|S_r|-1$}{\label{check3}
     $d(j+k-1) = |x_{S_r}(j+1)-x_{S_r}(j)|$\;}\label{end}
\end{algorithm}%

\subsection{Incidence Matrix Representation of Rainflow Algorithm}

We now illustrate how the Rainflow algorithm can be represented by the piece-wise linear map from the SoC profile $x$ to the cycle depth vector $d$:
\begin{equation}
    d = Rainflow(x)  = M(x)^T x.
    \label{rainflow_matrix_def}
\end{equation}
Here $M(x) \in \mathbb{R}^{(T+1)\times T}$ is an incidence matrix for a $x$-dependent directed graph $\mathcal{G}(x) := \mathcal{G}(x;\mathcal{V},\mathcal{E})$, with rows and columns representing nodes in $\mathcal{V}$ and edges in $\mathcal{E}$, respectively. We represent the $(i,j)^{\mbox{th}}$ element of $M(x)$ as $M_{ij}(x)$, or just $M_{ij}$ if its dependence on $x$ is clear from the context.
The graph $\mathcal{G}(x)$ consists of $T+1$ nodes, indexed by $t\in\{0,1,\dots,T\}$.

Algorithm~\ref{alg:incidence_matrix} provides the procedure for finding the edges of $\mathcal{G}(x)$, which is  summarized as follows.
\begin{itemize}
    \item Each full cycle identified 
    by Algorithm~\ref{alg:rainflow} corresponds to an element $(t_1,t_2)\in S_f$. For each of these cycles add $(t_1,t_2)$ to the edge set $\mathcal{E}$ \emph{twice}, i.e. $\mathcal{E} = \mathcal{E} \cup \{(t_1,t_2),(t_1,t_2)\}$, as outlined in
  steps~\ref{full_graph_start}-\ref{full_graph_end} in Algorithm~\ref{alg:incidence_matrix}.
    \item  Using the residual set $S_r$  
   output from Algorithm~\ref{alg:rainflow} and the direction operation $\textrm{dir}(x,S_r,j)$,
    iterate through $j = 1:|S_r|-1$ to add one directed edge, corresponding to each half-cycle to connect nodes $S_r[j]$ and $S_r[j+1]$. 
    See steps~\ref{rex_graph_start}-\ref{res_graph_end} in Algorithm~\ref{alg:incidence_matrix}.
\end{itemize}
We illustrate this for an example in left panel of Fig.  \ref{fig:graph_example}. Given $(x_0,\dots,x_5)$, there are two cycles with depth $x_3-x_2$ and $x_1-x_4$. The output of Algorithm \ref{alg:rainflow} would be $S_f=\{(3,2),(1,4)\}$, $S_r=\{0,5\}$, and $d=[x_3-x_2,x_3-x_2,x_1-x_4,x_1-x_4,x_5-x_0]^T$. This leads to $\mathcal{E}=\{(3,2),(3,2),(1,4),(1,4),(5,0)\}$, shown in the right panel of Fig.~\ref{fig:graph_example}.

\begin{algorithm}[ht]
\SetAlgoLined
\KwResult{Incidence Matrix $M(x)\in \mathbb{R}^{(T+1) \times T}$ }
\KwInput{SoC profile $x$, sets $S_r$ and $S_f$}
\KwInitialize{Digraph $\mathcal G(x;\mathcal{V},\mathcal{E})$, $\mathcal{V} = [0,\hdots,T],$ $\mathcal{E}= \emptyset$}
\For{i =1:$|S_f|$\label{full_graph_start}}{
    $\mathcal{E} = \mathcal{E} \cup\{S_f[i],S_f[i]\}$\;
    }
    \label{full_graph_end}
\For{j=1:$|S_r|-1$\label{rex_graph_start}}{
    $\mathcal{E} = \mathcal{E} + \textrm{dir}(x,S_r,j) $\;} 
    \label{res_graph_end} 
    Define $M(x)$ as the incidence matrix for $\mathcal G$ and attach zero columns as necessary.   
\caption{Rainflow Incidence Matrix M}
\label{alg:incidence_matrix}
\end{algorithm}
\begin{figure}[ht]
    \centering
    \includegraphics[width = 0.45\textwidth ]{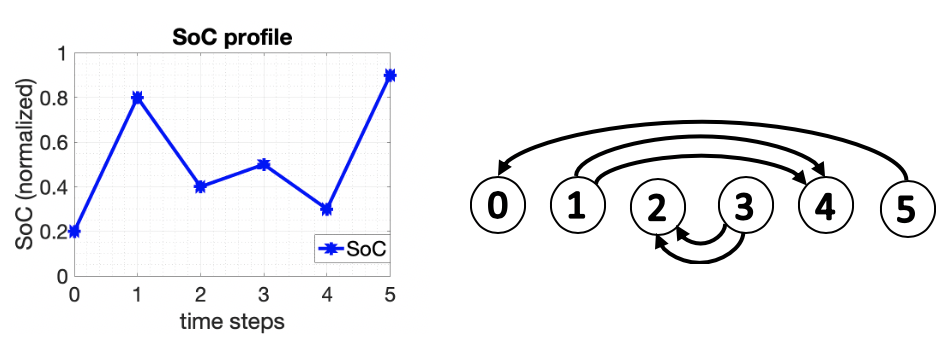}
    \caption{An example of SoC profile and its associated graph.}
    \label{fig:graph_example}
\end{figure}%
We specify the incidence matrix $M(x)$ of the graph $\mathcal{G}(x)$ such that the edges are indexed according to the order in which they are added. Then zero columns are attached to fill in the remainder of the incidence matrix such that $M(x)\in\mathbb{R}^{(T+1)\times T}$ always holds. For example, the incidence matrix for the example SoC profile in Fig.~\ref{fig:graph_example} is
\begin{align*}
M(x) = \begin{bmatrix} 0&0&0&0&-1\\0&0&1&1&0\\-1&-1&0&0&0\\1&1&0&0&0\\0&0&-1&-1&0\\0&0&0&0&1\end{bmatrix}.
\end{align*}
We can now explicitly express the total storage cycling cost \eqref{str_cost_original} in terms of the SoC profile $x$ as
\begin{eqnarray}
    C_s(x) = \frac{\alpha_bBE}{2}\sum_{i=1}^{T}(d_i)^{\beta_b} = \frac{\alpha_bBE}{2}\sum_{i=1}^{T}\left(\sum_{j=1}^{T+1}M_{ji}x_{j-1}\right)^{\beta_b}.
    \label{str_cost}
\end{eqnarray}
where we basically substitute~\eqref{rainflow_matrix_def} in~\eqref{str_cost_original}. \textcolor{black}{{Temporal coupling arises through the dependence of the incidence matrix $M(x)$ on the SoC  profile $x$.}}

\section{Structural Results} \label{Section IV}

\textcolor{black}{The piece-wise polynomial structure of the cost function \eqref{str_cost} makes it challenging to solve \eqref{optm_problem} even numerically.
} 
In this section we show that not only it is possible to efficiently solve \eqref{optm_problem}, but also by setting the price vector $p$ equal to the optimal Lagrange multiplier $\lambda^*$ associated with \eqref{power_bal_const}, the solution to \eqref{optm_problem} is also optimal  w.r.t. the generator and storage subproblems \eqref{subproblem1} and \eqref{subproblem2}, thus achieving incentive compatibility.
We end this section by discussing uniqueness of \eqref{subproblem2}.

\subsection{Convexity of Cycling Cost}

Our results build on the fact that under mild assumptions, the cycling cost function $C_s(x)$ in \eqref{str_cost} is convex~\cite{shi2017optimal}.
\begin{theorem}[Cycling Cost Convexity~\cite{shi2017optimal}]
For a given convex stress function $\Phi$, the cycling cost function $C_s(x)$ \eqref{str_cost} is convex with respect to the SoC profile $x$.
 \label{theorem1}
\end{theorem}

This striking result appeared in \cite{shi2017optimal} in the study of cycling-cost aware models for pay-for-performance storage operation. The proof relies on an implicit assumption in their induction method, which restricts combinations of two SoC profiles to those consisting of the same number of non-zero step changes. Extensions of that proof method rely on identifying and enumerating new scenarios. We avoid the need for this enumeration through an alternate proof method, which builds upon~\cite[Lemma~1]{shi2017optimal} to extend the applicability of their result.
\ifthenelse{\boolean{arxiv}}{The full proof is in the Appendix.}{Due to the page limitations we omit the proof here, it is provided in the technical report~\cite{BYGM20}.} An important consequence of Theorem \ref{theorem1} is that it makes our economic dispatch problem \eqref{optm_problem} convex. Therefore, off-the-shelf solvers can be applied to attain a globally optimal dispatch with the minimum total cost that strikes a trade-off between generation and degradation aware storage usage.
\subsection{Dual Pricing and Dispatch Optimality}

We now proceed to show how a proper choice of the price vector $p$ ensures that the optimal dispatch of \eqref{optm_problem} is also optimal for subproblems \eqref{subproblem1} and \eqref{subproblem2}. We use  $\lambda \in \mathbb{R}^T$ to denote the vector of dual variables associated with the power balance constraint \eqref{power_bal_const}. 
The (partial) Lagrangian for the problem \eqref{optm_problem} that only relaxes \eqref{power_bal_const} gives
\begin{align}
    \!\!\!  \mathcal{L}_1(g,u,x,\lambda) := \alpha_g g^Tg + \beta_g \mathbf{1}^T g +C_s + \lambda^T(D+u-g).
     \label{lagrangian_prob}
\end{align}
The corresponding dual problem is given by
\begin{eqnarray}
    && \max_\lambda \quad \mathcal{D}(\lambda) 
\label{dual_problem}
\end{eqnarray}
where for 
 a fixed $\lambda$
\begin{eqnarray}
 \mathcal{D}(\lambda) := \min_{g,u,x :  \eqref{gen_limit_ineq} \eqref{Soc_limit_ineq}\eqref{SoC_control_limit}\eqref{SoC_evolution}\eqref{SoC_terminal_const}}\mathcal{L}_1(g,u,x,\lambda) \nonumber  \ .
\end{eqnarray}
Note that the 
$\mathcal{D}(\lambda) $ is separable in terms of $g$ and $(u,x)$ as
\begin{equation}\label{decompoisition}
    \mathcal{D}(\lambda) = S_g(\lambda) + S_s(\lambda) + \lambda^T D,
\end{equation}
where $S_g(\lambda)$ and $S_s(\lambda)$ are respectively equivalent to the subproblems \eqref{subproblem1}, \eqref{subproblem2} with market clearing prices $p=\lambda$. 

This motivates setting the market clearing prices to $p=\lambda^*$, where $\lambda^*$ is the optimal solution of \eqref{dual_problem}. Intuitively, the dual variables $\lambda$ indicate the marginal cost of maintaining power balance. 
In fact, since all of the constraints in the primal problem \eqref{optm_problem} are affine, linear constraint qualifications are satisfied, guaranteeing strong duality \cite{boyd}. Hence the duality gap is zero and the primal problem \eqref{optm_problem} and the dual problem \eqref{dual_problem} coincides at the optimal values. This implies the optimality of the decomposition \eqref{decompoisition}. More precisely, the dual pricing scheme is incentive compatible, as we formalize below.
\begin{theorem}\label{theorem2}
Given an optimal primal-dual solution $(g^{*},u^{*},x^{*},\lambda^{*})$ to the economic dispatch problem \eqref{optm_problem}, $g^*$ and $(u^{*},x^{*})$ are also optimal with respect to the generator subproblem \eqref{subproblem1} and the storage subproblem \eqref{subproblem2}, respectively, given the market clearing prices $p=\lambda^*$.
\end{theorem}
\ifthenelse{\boolean{arxiv}}{
\begin{proof} 
Denote the dual variables associated with the constraints \eqref{SoC_evolution}, \eqref{SoC_terminal_const}, and \eqref{power_bal_const} as $\theta, \ (\omega_0,\omega_{T+1}), \mbox{ and }\lambda$, respectively. Further define $ (\overline{\mu},\underline{\mu}), \ (\underline{\nu},\overline{\nu}), \mbox{ and }(\overline{\gamma},\underline{\gamma})$ to be the non-negative dual variables associated with the inequality constraints \eqref{Soc_limit_ineq}\eqref{SoC_control_limit}\eqref{gen_limit_ineq}, respectively. The Lagrangian that relaxes all of the constraints is defined below:
\begin{eqnarray}
    &&\mathcal{L}(g,u,x,\lambda,\theta,\omega_{0},\omega_{T+1},\overline{\gamma},\underline{\gamma},\overline{\mu},\underline{\mu},\overline{\nu},\underline{\nu}) = C_s +\alpha_g g^Tg  \nonumber\\
    &&  + \beta_g \mathbf{1}^T g + \lambda^T(D+u-g)+ \theta^T(Ax-\frac{1}{E}u) + \omega_{0}(x^Te_1-x_o) \nonumber\\
    &&   +\omega_{T+1}(x^Te_{T+1}-x_o) + \overline{\gamma}^T(g-\overline{g})+\underline{\gamma}^T(\underline{g}-g) +\overline{\mu}^T(x-\mathbf{1}) \nonumber\\
    &&  +\underline{\mu}^T(-x) + \overline{\nu}^T(u-\overline{u})+\underline{\nu}^T(-\underline{u}-u).
\end{eqnarray}
The KKT conditions require 

stationarity:
\begin{subequations}
\begin{align}
    \left.\partial_g \mathcal{L}\right\rvert_{g=g^{*}}= & \ 2\alpha_g g^{*} + \beta_g\mathbf{1}  - \lambda^{*} + \overline{\gamma}^{*} - \underline{\gamma}^{*}= 0 \label{kkt_gen}\\
    \left.{\partial_x \mathcal{L}}\right\rvert_{x=x^{*}} = & \   \partial C_s(x^{*}) + A^T\theta^{*} + \omega_0^{*} e_1 \nonumber \\
    & \  + \omega_{T+1}^{*} e_{T+1} + \overline{\mu}^{*} - \underline{\mu}^{*} \ni 0 \label{kkt_str} \\
    \left.\partial_u \mathcal{L}\right\rvert_{u=u^{*}}= & \  \lambda^{*} - \frac{1}{E}\theta^{*}+\overline{\nu}^{*} - \underline{\nu}^{*} = 0 \label{kkt_str_control} 
\end{align}
\end{subequations}

\noindent
and complimentary slackness:
\begin{subequations}
\begin{eqnarray}
    &{\overline{\gamma}^{*}}^T (g^{*}  - \overline{g}) = 0 \quad  &{\underline{\gamma}^{*}}^T (\underline{g}- g^{*}) = 0  \label{complementary_gen}\\
    &{\overline{\mu}^{*}}^T (x^{*} - \mathbf{1}) = 0 \quad & {\underline{\mu}^{*}}^T x^{*} = 0 \label{complementary_str}\\
    & {\overline{\nu}^{*}}^T (u^{*} - \overline{ u}) = 0 & {\underline{\nu}^{*}}^T (-\underline{u} - u^{*}) = 0. \label{complementary_str_control}
\end{eqnarray}
\end{subequations}
Furthermore, the primal feasibility is given by the constraints \eqref{Soc_limit_ineq}-\eqref{SoC_terminal_const},\,\eqref{gen_limit_ineq},\,\eqref{power_bal_const} in the economic dispatch problem \eqref{optm_problem}, while the dual feasibility requires non-negativity of the dual variables associated with the inequality constraints in  \eqref{Soc_limit_ineq},\,\eqref{SoC_control_limit},\ and \eqref{gen_limit_ineq}. 

In terms of generator subproblem \eqref{subproblem1}, given $\lambda^*$, the conditions \eqref{gen_limit_ineq},\,\eqref{kkt_gen},\,\eqref{complementary_gen} and non-negativity of $(\overline{\gamma},\underline{\gamma})$ jointly account for its KKT conditions. Therefore, $g^*$ is also optimal with respect to generator subproblem \eqref{subproblem1}, given $\lambda^*$.
An analogous analysis applies to storage subproblem \eqref{subproblem2}.
\end{proof}
}
{{The proof is provided in \cite{BYGM20}.}} 
Theorem~\ref{theorem2} ensures that the dual pricing scheme is incentive compatible in the sense that given the market clearing prices, the market dispatch also maximizes the individual profits of the generator and the storage. In this way, all the participants will be fully incentivized.
\subsection{Uniqueness of Solution to Storage Subproblem}

In general, the storage subproblem \eqref{subproblem2} is not strictly convex. Therefore, its solution may not be unique, and the extension of this problem to potential distributed regimes of operation, control and market bidding is limited. We explore here conditions under which the solution to the storage subproblem is unique.

For ease of analysis, we use $\beta_b = 2$ as an approximation to the cycling cost coefficient $\beta_b$, which is empirically estimated to be $2.03$~\cite{shi2017optimal}. The cycling cost function in \eqref{str_cost} 
then reduces to
\begin{eqnarray}
    C_s(x) = \frac{\alpha_b BE}{2}x^TM(x)M(x)^Tx. \label{str_cost_simple}
\end{eqnarray}
We first propose a lemma that will enable us to obtain a sufficient condition for unique solution to storage subproblem.  
\begin{lemma}
Given any incidence matrix $N \in \mathbb{R}^{(T+1)\times T}$ of rank $r$, $\mathrm{rank}([NN^T \;e_{T+1}^T\;e_1^T]^T) = \mathrm{rank}([N^T\;e_{T+1}^T\;e_1^T ]^T) = r'$ holds with $r'$ bounded as $\min \{r+1,T+1\} \leq r'\leq \min \{r+2,T+1\}$.
Furthermore, if the rank of the incidence matrix $M(x)$ in \eqref{rainflow_matrix_def} is $r$, $\mathrm{rank}([M(x)M(x)^T e_{T+1}^T e_1^T]^T) = r+1 $.
\label{lemma1}
\end{lemma}

The proof uses the formula for the rank of a product of matrices and the fact that the standard basis vectors ($e_1$ and $e_{T+1}$) can add at least one and at most two independent rows to the conjoined matrix.

The following theorem characterizes the sufficient condition for a unique solution to the storage subproblem \eqref{subproblem2}.
\begin{theorem}\label{theorem3}
 Assume that storage cycling cost function $C_s(x)$ takes form of \eqref{str_cost_simple} and is differentiable in small neighbourhood of $x^*$, where $x^*$ is optimal for  the economic dispatch problem \eqref{optm_problem}. Assume that all inequality constraints are satisfied strictly. If the rank of the incidence matrix $M(x^*) 
 $ is $T$, then the storage subproblem has a unique solution. 
\end{theorem}
\ifthenelse{\boolean{arxiv}}{
\begin{proof}
In this proof, the subscript $s$ refers to the solution to the 
storage subproblem \eqref{subproblem2}. 
Using the assumption that the cycling cost function~\eqref{str_cost_simple} is differentiable in a small neighbourhood of $x^{*}$, $M(x^{*})$ is fixed in that neighbourhood. 
We define the Lagrangian of the economic dispatch problem \eqref{optm_problem} to be 
\begin{eqnarray}
    &&\mathcal{L}(g,u,x,\lambda,\theta,\omega_{0},\omega_{T+1})= \alpha_g g^Tg + \beta_g \mathbf{1}^T g \nonumber\\
    &&  +\frac{\alpha_b BE}{2}x^TM(x)M(x)^Tx + \lambda^T(D+u-g)+ \theta^T(Ax-\frac{1}{E}u) \nonumber\\
    &&  + \omega_{0}(x^Te_1-x_o) +\omega_{T+1}(x^Te_{T+1}-x_o)
\end{eqnarray}
The KKT conditions can be written as

\noindent
stationarity:
\begin{align}
    \left.\partial_g \mathcal{L}\right\rvert_{g=g^{*}} = & \ 2\alpha_g g^{*} + \beta_g\mathbf{1}  - \lambda^{*} = 0 \nonumber \\
   \left.\partial_x \mathcal{L}\right\rvert_{x=x^{*}} = & \ \alpha_b BE M(x^{*})M(x^{*})^T x^{*} + A^T\theta^{*}  \nonumber \\
   &+ \omega_{0}^{*}e_1+\omega_{T+1}^{*}e_{T+1} = 0  \nonumber  \\
     \left.\partial_u \mathcal{L}\right\rvert_{u=u^{*}} = & \ \lambda^{*} - \frac{1}{E}\theta^{*} = 0 \nonumber
\end{align}

\noindent
primal feasibility:
\begin{align}
     &\left.\partial_{\lambda} \mathcal{L}\right\rvert_{\lambda=\lambda^{*}} = \ D + u^{*} - g^{*} = 0\nonumber\\
      &\left.\partial_{\theta} \mathcal{L}\right\rvert_{\theta=\theta^{*}}  = \ Ax^{*} -  \frac{1}{E}u^{*} =0 \nonumber\\
      &\left.\partial_{\omega_{0}} \mathcal{L} \right\rvert_{\omega_0=\omega_0^{*}} = \ e_1^Tx^{*} - x_o =0  \nonumber\\
      &\left.\partial_{\omega_{T+1}} \mathcal{L}\right\rvert_{\omega_{T+1}=\omega_{T+1}^{*}} = \  e_{T+1}^Tx^{*} - x_o = 0 \nonumber
\end{align}
which imply
\begin{equation}
     \lambda^{*} = 2\alpha_g (D+EAx^{*}) + \beta_g\mathbf{1} \label{kkt_mkt_lmp} ,
\end{equation}
and 
\begin{align}
     & \ \alpha_b BE M(x^{*})M(x^{*})^Tx^{*} + 2E^2\alpha_g A^TAx^{*} + \omega_{0}^{*} e_1  +\omega^{*}_{T+1} e_{T+1}  \nonumber \\
      &= \ -\beta_g EA^T\mathbf{1} -2\alpha_g EA^TD  . \label{kkt_mkt}
\end{align}
\eqref{kkt_mkt_lmp} and \eqref{kkt_mkt} are the sets of equations to solve for the optimal market clearing prices and the optimal storage SoC profile $x^{*}$ from the market perspective, respectively. 

The Lagrangian for the storage subproblem is defined as 
\begin{eqnarray}
    &&\mathcal{L}_s(u_s,x_s,\theta_s,\omega_{0,s},\omega_{T+1,s}) =\frac{\alpha_b BE}{2}x_s^TM(x_s)M(x_s)^Tx_s   \nonumber\\
    &&  + \lambda^Tu_s + \theta_s^T(Ax_s-\frac{1}{E}u_s) + \omega_{0,s}(e_1^Tx_s-x_o) \nonumber \\
    && +\omega_{T+1,s}(e_{T+1}^Tx_s-x_o)
\end{eqnarray}
The KKT conditions for the storage subproblem can be similarly written as

\noindent
stationarity:
\begin{align}
     \left.\partial_{x_s} \mathcal{L}_s\right\rvert_{x_s=x_s^{*}}=  &\ \alpha_b B E M(x_s^{*})M(x_s^{*})^Tx^{*}_s + A^T\theta^{*}_s \nonumber\\ &+\omega_{0,s}^{*}e_{1}+\omega_{T+1,s}^{*}e_{T+1}  = 0  \nonumber\\
    \left.\partial_{u_s} \mathcal{L}_s \right\rvert_{u_s=u_s^{*}} =  &\ \lambda^{*} - \frac{1}{E}\theta^{*}_s = 0 \nonumber
\end{align}

\noindent
primal feasibility:
\begin{align}
    &\left.\partial_{\theta_s} \mathcal{L}_s\right\rvert_{\theta_s=\theta_s^{*}}=  \ Ax^{*}_s - \frac{1}{E}u^{*}_s =0 \\
    &\left.\partial_{\omega_{0,s}} \mathcal{L}_s\right\rvert_{\omega_{0,s}=\omega_{0,s}^{*}} =  \ e_1^Tx^{*}_s - x_o =0 \label{sub_kkt_terminal 1}\\
    &\left.\partial_{\omega_{T+1,s}} \mathcal{L}_s\right\rvert_{\omega_{T+1,s}=\omega_{T+1,s}^{*}}=  \  e_{T+1}^Tx^{*}_s - x_o =0 \label{sub_kkt_terminal 2}
\end{align}
Substituting $\lambda^{*}$ from \eqref{kkt_mkt_lmp} in the above, we attain
\begin{align}
     \alpha_b BE  M(x^{*}_s)M(x^{*}_s)^Tx^{*}_s =& - 2E\alpha_g A^TAx^{*} -\omega_{0,s}^{*}e_{1} -\omega_{T+1,s}^{*}e_{T+1} \nonumber\\ & -\beta_g EA^T\mathbf{1} -2\alpha_g EA^TD , \label{kkt_str_sub}
\end{align}
which, combined with the primal feasibility constraint \eqref{sub_kkt_terminal 1}\eqref{sub_kkt_terminal 2}, yields the set of equations to solve for $x^{*}_s$:
\begin{equation}
    \begin{bmatrix}\alpha_b BE M(x^{*}_s)M(x^{*}_s)^T\\e_{T+1}^T\\e_1^T \end{bmatrix}x^{*}_s = \begin{bmatrix} h(x^{*},\omega^{*}_{0,s},\omega^{*}_{T+1,s})\\ x_o\\x_o\end{bmatrix} \label{sub_clsd_frm},
\end{equation}
where $h(x^{*},\omega^{*}_{0,s},\omega^{*}_{T+1,s})$ summarizes the R.H.S of \eqref{kkt_str_sub}.
Given Theorem \ref{theorem2}, $x^*$ from the market economic dispatch problem should always be a solution to \eqref{sub_clsd_frm}. In that case, Lemma~\ref{lemma1} suggests the matrix on the L.H.S of \eqref{sub_clsd_frm} is full column rank as the rank\footnote{Note that $M(x^{*}_s)$ of rank $T$ is unique.} of matrix $M(x^{*}_s)$ will be $T$,  
which guarantees a unique $x^*_s$.
Therefore, $x_s^*$ is unique for the storage subproblem. 
\end{proof}
}
{{The proof is provided in \cite{BYGM20}.}} 
The cases in which Theorem~\ref{theorem3} holds are restrictive in the sense that there is no cycle in the SoC profile $x$ and the profile switches between charging and discharging half-cycles at every time slot. 
It suggests that in general market price signals are not sufficient to align individual participant incentives with economic dispatch objective in fully distributed regimes.
\textcolor{black}{The above results can also be generalized to economic dispatch problems with multiple generators and storage units, due to the convexity cycling cost.}


\section{Numerical Simulation} \label{Section V}

In this section we present numerical results using aggregate demand data for a single day for one zone operated by the NYISO as an illustrative example (date: 3/9/2020, Zone H)~\cite{nyisodata}. For the generator in our setup, we use the cost coefficients $\alpha_g = 0.1$ and $\beta_g = 20$ in equation \eqref{optm_obj} that correspond to the average cost coefficients from the IEEE 300-bus system~\cite{pyou,matpower}.
We assume that the generator has sufficient capacity to meet the peak demand, i.e. $\overline{g} \ge \max_t \{D_t\}$ with $\underline{g} = 0$.
The storage cycling cost coefficients are set to $\alpha_b = 5.24 \times 10^{-4}$ and $\beta_b = 2.03$, which correspond to empirically determined values based on historical data \cite{shi2017optimal}. 
The power rating of the storage is given by 
$\overline{u} = \frac{E}{4}$ and $\underline{u} = -\frac{E}{4}$.

We compute results for three dispatch strategies \textcolor{black}{to gain more insight into the importance of accounting for storage degradation in economic dispatch}:
\begin{itemize}
    \item Generation Centric Dispatch (\emph{GCD}): Co-optimization of generation and storage operation that accounts for only generation cost i.e. eliminating storage degradation cost from the objective function in~\eqref{optm_obj}. This leads to a dispatch strategy that is unaware of the cycling cost associated with storage use (i.e. storage degradation is a hidden cost that is computed afterwards from the optimal storage profile). We then define the total cost = generation cost + \emph{hidden} cycling cost\footnote{These hidden costs may represent the notion of uplift payment necessary to incentivize storage participation.};
    \item Storage Degradation Aware Dispatch (\emph{SDAD}): Co-optimization of generation and storage operation that accounts for both generation cost and storage cycling cost, i.e. problem \eqref{optm_problem}, in this case total cost = generation cost + cycling cost;
    \item Generator Dispatch (\emph{GD}): Optimization of the generator profile i.e. there is no storage and the total cost = generation cost. 
\end{itemize}

\begin{figure}[ht]
    \centering
    \includegraphics[width=0.48\textwidth]{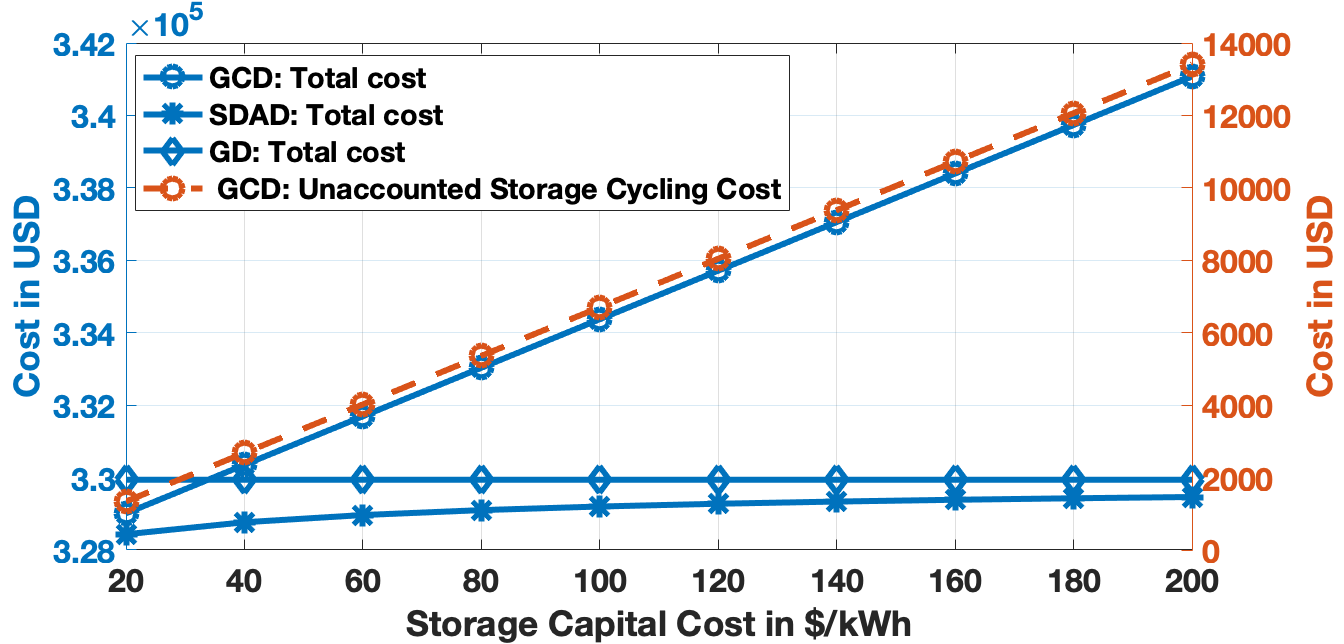}
     \caption{Total cost of GCD, SDAD and GD, and Unaccounted storage cycling cost of GCD w.r.t. storage capital cost. }
     \label{fig:cost_vs_B}
\end{figure}

Fig.~$\ref{fig:cost_vs_B}$ compares the total cost of the dispatch strategies as we increase the storage capital cost given a fixed storage capacity. Here we fix the storage capacity to be $E = 500 MWh$ ($7.65\%$ storage penetration w.r.t. daily energy demand). As expected, our approach SDAD gives the minimum total cost amongst the three, while GCD performs worst, especially as the storage capital costs increase, as this increases the hidden cycling costs that are not taken to account for in the dispatch decisions. The dashed line curve in Fig.~\ref{fig:cost_vs_B} right y-axis more explicitly shows how the unaccounted for storage cycling cost deteriorates the performance of GCD.

\begin{figure}[ht]
    \centering
    \includegraphics[width=0.48\textwidth]{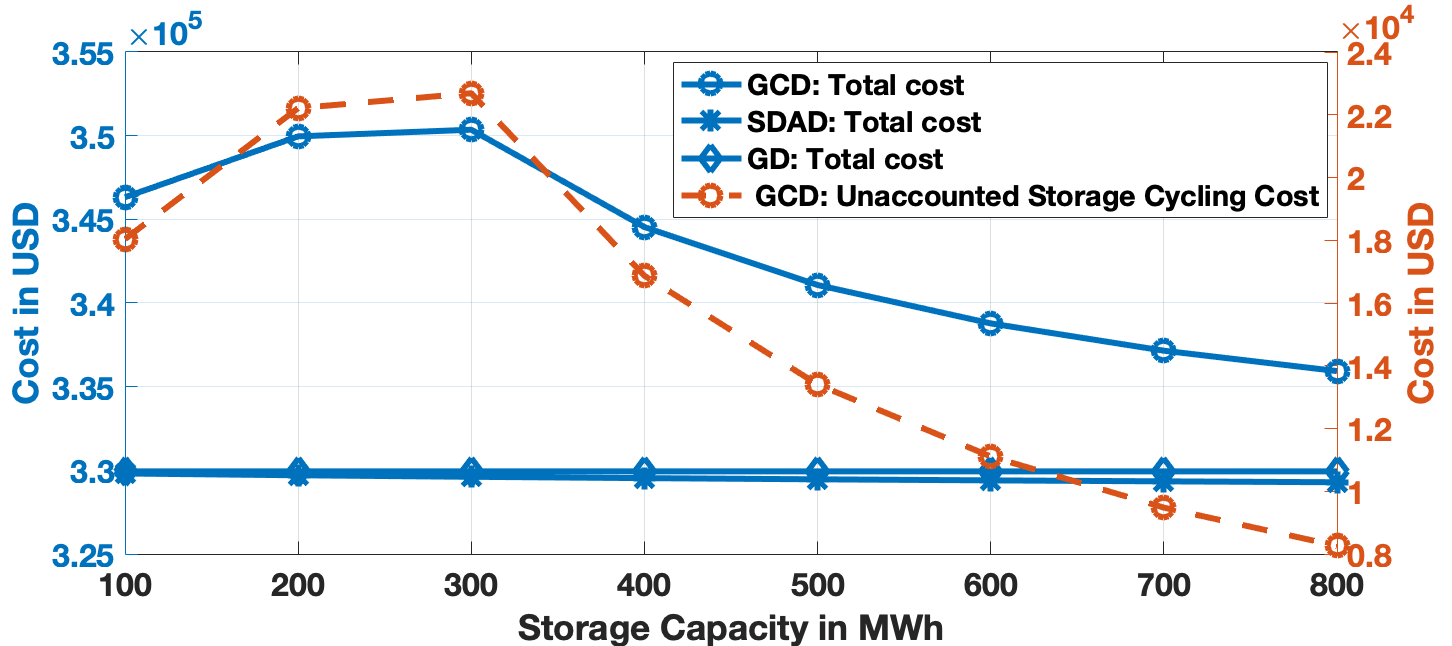}
    \caption{Total cost of GCD, SDAD and GD, and Unaccounted storage cycling cost of GCD w.r.t. storage capacity. }
    \label{fig:cost_vs_E}
\end{figure}

Fig.~\ref{fig:cost_vs_E} illustrates the impact of storage capacity on the total cost of the dispatch strategies with fixed capital cost $B$. In these results we use $B = 200 \$/kWh$, which corresponds to current estimated lithium-ion battery costs~\cite{mongird2019energy}. Despite the small difference in the total cost between SDAD and GD, the difference tends to increase with increasing storage capacity, meaning more savings with larger storage units. The total cost of SDAD decreases since the storage is able to supply the required power with shallower cycle depths, thus incurring lower degradation cost.  
In this comparison, GCD is again worst amongst all, though it improves as storage capacity grows, since the storage can supply the same amount of power with fewer cycles and shallower cycle depths. 
The unaccounted storage cycling cost for GCD is shown with dashed line curve in Fig.~\ref{fig:cost_vs_E} right y-axis. 
\begin{figure}[ht]
    \centering
    \includegraphics[width=0.45\textwidth]{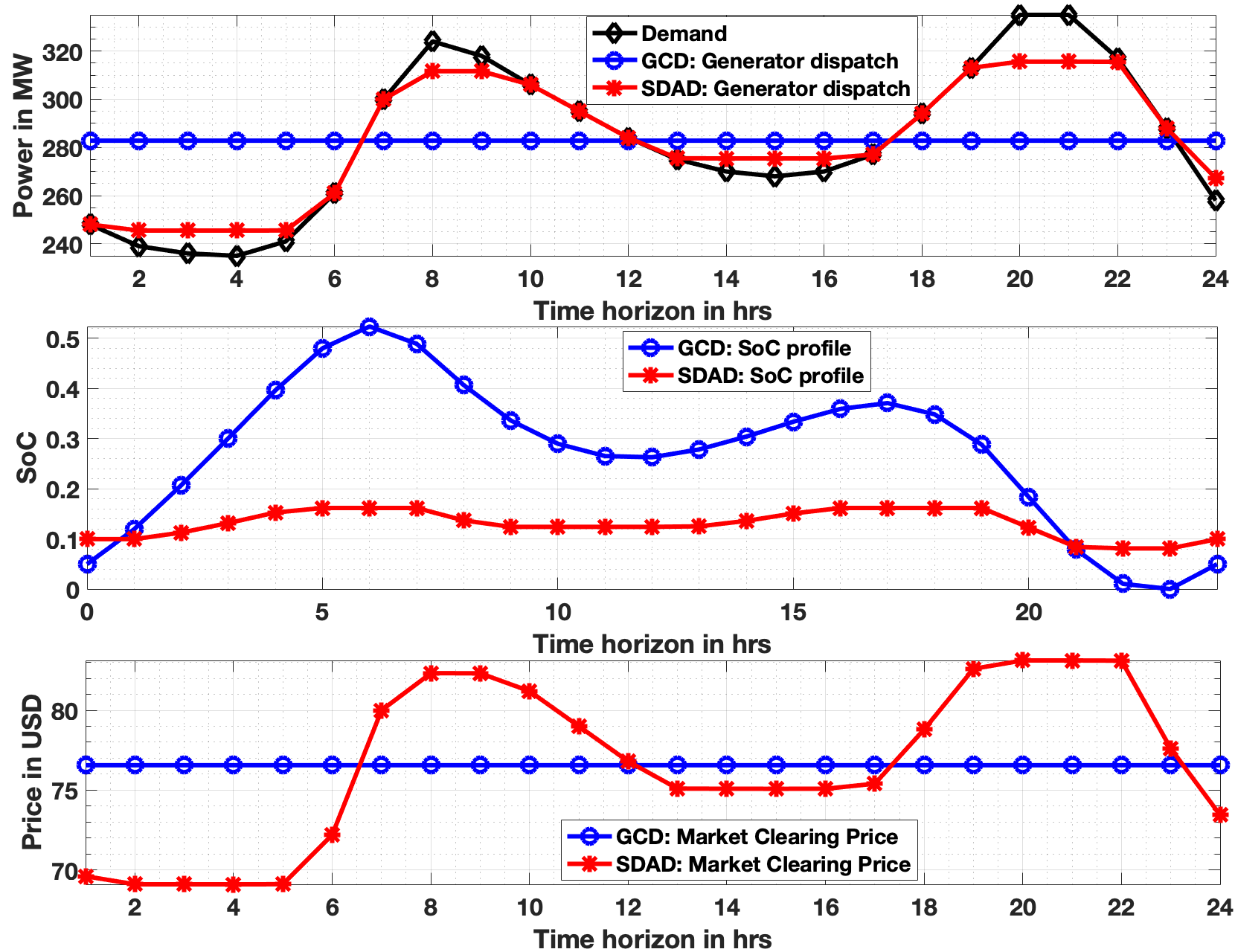}
    \caption{(Top) Demand and generation schedule of GCD, and SDAD; (Middle) SoC profile of GCD, and SDAD; (Bottom) Market clearing price of GCD, and SDAD}
    \label{fig:soc_lambda}
\end{figure}

We now fix both storage capacity to $E = 500 MWh$, and storage capital cost to $B = 200$/kWh and examine the optimal dispatch obtained under GCD and SDAD. The top and middle panels in Fig.~\ref{fig:soc_lambda} compare the optimal dispatch profile for generator and storage, respectively. As expected, our approach SDAD gives shallower depths due to degradation cost of storage while GCD utilizes storage without any restrictions. The market clearing price for GCD is flat due to this unrestricted utilization and sufficient storage capacity in the market. On the other hand SDAD results in time-varying market clearing prices as shown in the bottom panel in Fig.~\ref{fig:soc_lambda} because it takes into account the storage degradation as well as the generation costs.

These examples illustrate the need for an approach like SDAD that accounts properly for the hidden cycling cost of storage in order to take into account for the total cost of the system dispatch. 



\section{Conclusions} \label{Section VI}

In this paper, we formulate and analyze an economic dispatch problem that intrinsically accounts for cycle-based storage degradation costs. Using the convexity of this cycling cost, we show that the optimal economic dispatch along with dual pricing is incentive compatible, i.e., individual participants  attain maximum profits given market clearing prices. Further, with a digraph interpretation of the Rainflow algorithm, we provide a rank condition on the graph incidence matrix, which linearly maps a SoC profile to half-cycle depths, to guarantee the uniqueness of the optimal storage response to market clearing prices.
Numerical examples illustrate that accounting for storage degradation addresses a potential market inefficiency for storage participation due to the large unaccounted for storage operational costs in traditional economic dispatch formulations.

\ifthenelse{\boolean{arxiv}}{

\appendices
\begin{section}
{Proof of Theorem~\ref{theorem1}}
We present four lemmas and one proposition first. 
We define two vectors $d_c\in \mathbb{R}^T$ and $d_d\in \mathbb{R}^T$ based on depth vector $d$, such that $d_c$ contains charging half-cycle depths in decreasing order and $d_d$ contains discharging half-cycle depths in decreasing order. We append zeros in the tail to fill in the vectors as necessary. Accordingly, we can define digraphs $\mathcal{G}_c$ and $\mathcal{G}_d$ for charging and discharging half-cycles, respectively, based on $\mathcal{G}$, and the associated incidence matrices $M_c(x) \in \mathbb{R}^{(T+1)\times T}$ and $M_d(x)\in \mathbb{R}^{(T+1)\times T}$ based on $M(x)$, attached with zero columns as necessary, such that $M_c(x)^T x = d_c $ and $M_d(x)^Tx =d_d$. Furthermore, we rearrange elements (rows) of depth vector $d$ (incidence matrix $M(x)$) in decreasing order to maintain $M(x)^T x = d$. As an illustrative example, the incidence matrices $M_c(x)$ and $M_d(x)$ for the SoC profile in Fig.~\ref{fig:graph_example} is 
\begin{footnotesize}
\begin{align*}
M_c(x) = \begin{bmatrix} -1&0&0&0&0\\0&1&0&0&0\\0&0&-1&0&0\\0&0&1&0&0\\0&-1&0&0&0\\1&0&0&0&0\end{bmatrix} , \
M_d(x) = \begin{bmatrix} 0&0&0&0&0\\1&0&0&0&0\\0&-1&0&0&0\\0&1&0&0&0\\-1&0&0&0&0\\0&0&0&0&0\end{bmatrix} .
\end{align*}
\end{footnotesize}
%

We will focus our analysis on charging half-cycles, including residual half-cycles and those originated by full cycles. The analysis for discharging half-cycles is analogous. 
Recall the full cycle set $S_f$ of \emph{time index pairs} and the residue set $S_r$ of \emph{time indices} that Algorithm~\ref{alg:rainflow} outputs. 
Also, for any edge $(u,v)$ of a digraph, node $u$ is at the tail of the edge, and node $v$ is at the head of the edge. A source (sink) is defined as a node of a digraph, which is at the tail (head) of all the edges associated with that node~\cite{FB-LNS}. Each time node $t \in  \{0,1,...,T\}$ is a source, or a sink, of at most one charging half-cycle. 
Therefore, 
$[M_c \mathbf{1}]_i \in  \{-1,0,+1\} $ holds for $\forall  i \in \{1,2,...,T+1\}$. 
\begin{definition}
A directed edge between nodes $t_1$ and $t_4$ is said to envelop another edge between nodes $t_2$ and $t_3$, if $t_1<t_2<t_3<t_4$ holds, where $t_i \in \{0,1,...,T\}\ \forall i \in \{1,2,3,4\}$ are four nodes of the digraph.
\end{definition}
\begin{remark}
Given $(\Delta_{j-1},\Delta_j,\Delta_{j+1})= \diff(x,S_r,j)$, assume that a full cycle exists, i.e., $\Delta_{j-1} \geq \Delta_{j}$, and $\Delta_{j+1} \geq \Delta_{j}$ holds. During the extraction of the full cycle, if the corresponding full cycle 
represented by the time indices $\{S_r[j],S_r[j+1]\}$ 
follows $x_{S_r}(j) > x_{S_r}(j+1)$, then the remaining profile denoted by the time indices $\{S_r[j-1],S_r[j+2]\}$ 
satisfies $x_{S_r}(j-1) < x_{S_r}(j+2)$, and vice versa. To see this, notice
\begin{align}
    &x_{S_r}(j) > x_{S_r}(j+1) \implies \Delta_j= -(x_{S_r}(j+1)-x_{S_r}(j)),\\
    &\Delta_{j-1}=x_{S_r}(j) - x_{S_r}(j-1)>0, \quad \text{ and }\\
    &\Delta_{j+1}=x_{S_r}(j+2) - x_{S_r}(j+1)> 0  
\end{align}
Thus, by using the fact that $\Delta_j$ was extracted we have
\begin{align*}
 x_{S_r}(j) - x_{S_r}(j-1) &\geq-(x_{S_r}(j+1)-x_{S_r}(j))\\
 x_{S_r}(j+2) - x_{S_r}(j+1)&\geq-(x_{S_r}(j+1)-x_{S_r}(j))
\end{align*}
The first one leads to $x_{S_r}(j-1)\leq x_{S_r}(j+1)$ and the second one $ x_{S_r}(j+2)\geq x_{S_r}(j)$. Thus, 
$$x_{S_r}(j-1)\leq x_{S_r}(j+1)<x_{S_r}(j)\leq x_{S_r}(j+2),$$
where the strict inequality is by assumption.

In other words, the half-cycle edge of the remaining profile envelop the associated full cycle edges in the opposite direction with respect to full cycle edges. 
\end{remark}



\begin{remark}\label{step_decomposition}
Any SoC profile $x$ can be written in terms of step functions, i.e. $x = \sum_{t = 0}^{T} p_{t}\mathbf{1}_{t}$ where, $p_{t} \in  \mathbb{R}$ is the amplitude and $\mathbf{1}_{t} \in  \mathbb{R}^{(T+1)}$ is defined as
\begin{align*}
    [\mathbf{1}_{t}]_i =
\left\{
	\begin{array}{ll}
		0  & \mbox{if } i < t \\
		1 & \mbox{if } i \ge t
	\end{array}
\right.
\end{align*}
\end{remark}

\begin{lemma}
\label{mcxlemma}
For any matrix $M_c(x)\in\mathbb{R}^{(T+1)\times T}$, the following holds,
\begin{align}
    \abs{(M_c(x)\mathbf{1})^T\mathbf{1}_{t}} \le 1, \ t \in \{0,1,...,T\} 
\end{align}
\end{lemma}
\begin{proof}
For any charging half-cycle edge between two time nodes $t_s$ and $t_e$, where the directed edge is either $(t_s,t_e)$, or $(t_e,t_s)$, assume w.l.o.g. $t_e > t_s$.  
Now, consider two such consecutive charging half-cycles edges between time nodes $t_{s,1},t_{e,1},$ and $t_{s,2},t_{e,2}$ respectively, then there are three possible cases,

\begin{itemize}
    \item Case 1: $t_{s,1}<t_{e,1}<t_{s,2}<t_{e,2}$, i.e. charging half-cycle edges are sequential to each other. 
    In this case, both edges should have same direction, as opposite direction is possible only when one charging half-cycle envelop another charging half-cycle.  Therefore,  $sgn([M_c\mathbf{1}]_i) = -sgn( [M_c \mathbf{1}]_{i+1}), \ i \in \{t_{s,1},t_{e,1},t_{s,2},t_{e,2}\}$, so $\abs{(M_c(x)\mathbf{1})^T\mathbf{1}_{t}} \le 1, \ t \in \{t_{s,1},t_{e,1},t_{s,2},t_{e,2}\} $.
    
    \item Case 2: $t_{s,1}<t_{s,2}<t_{e,1}<t_{e,2}$, i.e. charging half-cycle edges cross each other. 
    Charging half-cycle edges should have same direction, as the half-cycle doesn't envelop each other. Furthermore, by contradiction, if both have same direction, then $\left(x(t_{s,1}), x(t_{s,2}), x(t_{e,1})\right)$ is monotonic, and $t_{s,2}$ cannot be part of any half-cycle edge, as in step 1-5 in Algorithm~\ref{alg:rainflow}, which is contradictory to the fact that $\{t_{s,2},t_{e,2}\}$ is a charging-half cycle. Hence, this case is infeasible.  
    
    \item Case 3: $t_{s,1}<t_{s,2}<t_{e,2}<t_{e,1}$, i.e. one charging half-cycle envelop another charging half-cycle. By contradiction, if both have same direction, then $\left(x(t_{s,1}), x(t_{s,2}), x(t_{e,2})\right)$ is monotonic, and $t_{s,2}$ cannot be part of any half-cycle edge, as in step 1-5 in Algorithm~\ref{alg:rainflow}, which is contradictory to the fact that $\{t_{s,2},t_{e,2}\}$ is a charging-half cycle. Therefore both edges have opposite direction, and, $sgn([M_c\mathbf{1}]_i) =  -sgn( [M_c \mathbf{1}]_{i+1}), \ i \in \{t_{s,1},t_{e,1},t_{s,2},t_{e,2}\}$, so $\abs{(M_c(x)\mathbf{1})^T\mathbf{1}_{t}} \le 1, \ t \in \{t_{s,1},t_{e,1},t_{s,2},t_{e,2}\} $.
\end{itemize}
From above, any two consecutive charging half-cycles cannot intercept and the result could generalize with both cases 1 and 3 for any possible combination of charging half-cycles.
\end{proof}

\begin{definition}
A SoC profile $x$ is a boundary profile, if $\nexists$ $\epsilon > 0$ such that, $\forall$  
$\abs{q_t}\le \epsilon, \ t \in \{0,1,...,T\}$, the relation $M(x+ q_t\mathbf{1}_t) = M(x)
$ holds.
\end{definition}
\begin{remark}
For any arbitrary non-boundary profile $x$, $\exists$ $\epsilon > 0$, such that $\forall$ $\abs{q_t}<\epsilon, \ t \in \{0,1,...,T\}$, the relation $M(x+ q_t\mathbf{1}_t) = M(x)$ holds . 
\end{remark}

\begin{lemma}
Given a boundary profile $x$, $\exists$  $\epsilon>0$ small enough, such that $\forall$ $\abs{q_t}\le \epsilon, \ t \in \{0,1,...,T\}$, the following relation holds,
\begin{align}
 & M(x+q_t \mathbf{1}_t) ^Tx = M(x)^Tx 
\end{align}
\end{lemma}

\begin{proof}
A boundary profile seamlessly transitions between forming and not forming a full cycle. This transition is characterized by either of the following conditions. Given $(\Delta_{j-1},\Delta_j,\Delta_{j+1})= \diff(x,S_r,j)$, 
\begin{itemize}
    \item $\Delta_j = \Delta_{j-1}$, or $\Delta_j = \Delta_{j+1}$ $\implies$ $x_{S_r}(j-1)=x_{S_r}(j+1)$, or, $x_{S_r}(j)=x_{S_r}(j+2)$. WLOG, assume $\Delta_j = \Delta_{j-1}$, then any infinitesimally small change $ q_t, \ \abs{q_t} \le \epsilon, $ can alter $\Delta_j$, such that, either $\Delta_j < \Delta_{j-1}$, and this cycle of depth $\Delta_j$ can still be extracted as in steps 10-15 in Algorithm~\ref{alg:rainflow}, or $\Delta_j > \Delta_{j-1}$, and this cycle cannot be extracted anymore. Even though, the depth vector will be different for $\Delta_j < \Delta_{j-1}$ as in step 13 in Algorithm~\ref{alg:rainflow}, and for $\Delta_j > \Delta_{j-1}$ as in step 19 in Algorithm~\ref{alg:rainflow}. In the boundary case of $\Delta_j =\Delta_{j-1}$, as $x_{S_r}(j-1)=x_{S_r}(j+1)$, the depth vector can be written in any of the two forms as in step 13, or step 19 in Algorithm~\ref{alg:rainflow}, by substituting $x_{S_r}(j-1)$ with $x_{S_r}(j+1)$ and vice versa. 
    \item $\Delta_j = 0$ $\implies$ $x_{S_r}(j)=x_{S_r}(j+1)$. In this case, any infinitesimally small change $ q_t, \ \abs{q_t} \le \epsilon, $ can alter $\Delta_j$, such that, either $\Delta_j >0$, and the cycle of depth $\Delta_j$ can be extracted as in steps 10-15 in Algorithm~\ref{alg:rainflow}, or $(x_{S_r}(j-1),x_{S_r}(j),x_{S_r}(j+1))$ is monotonic, and this cycle cannot be extracted anymore. In the boundary case of $\Delta_j = 0$, since $(x_{S_r}(j-1),x_{S_r}(j),x_{S_r}(j+1))$ is monotonic, no cycle is extracted, which is equivalent to extracted cycle of depth $0$, as $x_{S_r}(j)=x_{S_r}(j+1)$ holds at the boundary.  
\end{itemize}
Despite $M(x+ q_t\mathbf{1}_t) \neq M(x)$ in the limit $ q_t  \rightarrow 0$, $M(x+q_t \mathbf{1}_t) ^Tx = M(x)^Tx$ holds for the boundary profile $x$ due to the exactly same depth vectors.  
\end{proof}

\begin{lemma}
Consider two SoC profiles $x$ and $y$ such that $M_c(y) = M_c(x)$ and $y = x + {q}_t\mathbf{1}_t,$ where $q_t$ is nonzero, $t \in \{0,...,T\}$. 
The following relations always hold,
\begin{small}
\begin{eqnarray}
    &\abs{\Delta d_{c,i}} \leq \abs{q_t}, \ \forall i \in \{1,...,T\}\\
        &\abs{ \sum_{i=1}^T \Delta d_{c,i}} \leq \abs{q_t}
\end{eqnarray}
\end{small}%
with $d_{c,i}(x) := [M_c(x)^Tx]_i$ and $\Delta d_{c,i}(y,x):= d_{c,i}(y)-d_{c,i}(x) = [M_c(y)^Ty]_i - [M_c(x)^Tx]_i$.

\label{lemma_3}
\end{lemma}
\begin{proof}: 
For the first result,
\begin{small}
\begin{align*}
    &\Delta d_c =  M_c(y)^{T}y- M_c(x)^{T}x \\
    \implies &\Delta d_c =  M_c(x)^{T}(y-x), \ \textrm{from assumption}\\
    \implies &\Delta d_c =  M_c(x)^{T}(q_t \mathbf{1}_t)\\
    \implies &\Delta d_c =   q_t(M_c(x)^{T}\mathbf{1}_t)\\
    \implies &\Delta d_{c,i} =   q_t\sum_{j=1}^{T+1} M_c(x)_{ij}^T\mathbf{1}_{t,j}\\
    \implies &\Delta d_{c,i} =   q_t\sum_{j=t}^{T+1} M_c(x)_{ji}\\
    \implies &\abs{\Delta d_{c,i}} =  \abs{ q_t\sum_{j=t}^{T+1} M_c(x)_{ji}}\\
   \implies &  \abs{\Delta d_{c,i}} \leq  \abs{q_t}, \ \forall i \in \{1,...,T\}
\end{align*}
\end{small}%
where the last inequality holds due the fact 
$(M_c(x)^{T}\mathbf{1}_t)_{i} \in  \{-1,0,+1\}$. 
Further, the second result holds as follows
\begin{align*}
    \abs{\sum_{i=1}^T \Delta d_{c,i}} = \ & \abs{\mathbf{1}^T\Delta d_c} = \abs{\mathbf{1}^TM_c(x)^T(y-x)}\\
    = & \abs{(M_c(x)\mathbf{1})^T( q_t \mathbf{1}_t)}\\
    = &  \abs{q_t} \abs{(M_c(x)\mathbf{1})^T\mathbf{1}_t}\\
    \implies \abs{\sum_{i=1}^T \Delta d_{c,i}} \le &  \abs{q_t}, \ \textrm{using Lemma~\ref{mcxlemma} }
\end{align*}
This completes the proof.
\end{proof}

Now, we will provide bounds on the change in the depth vector $d_c=M_c(x)^Tx$ of SoC profile $x$, for an arbitrary step change perturbation denoted by  $q_t\mathbf{1}_t, \ t \in \{0,1,...,T\}$.

\begin{proposition}\label{proposition_1} Consider a step change $q_t\mathbf{1}_t$ added to SoC profile $x$, s.t. $y = x + q_t\mathbf{1}_t,$ where $q_t$ is nonzero and $t \in \{0,...,T\}$. The Rainflow incidence matrix (only considering charging half-cycle depths) for $x$ and $y$ are given by $M_c(x)$ and $M_c(y)$, respectively. 
Then the following holds,
\begin{eqnarray}
    &\abs{\Delta d_{c,i}} \leq  \abs{q_t}, \forall i \in \{1,...,T\} \label{prop6.2}\\
    &\abs{ \sum_{i = 1}^{T} \Delta d_{c,i}} \leq \abs{q_t} \label{prop6.1} 
    \end{eqnarray}
\end{proposition}
\begin{proof}: The step change is split into parts, i.e., $\sum_{i=1}^n \Delta q_{i,t} = q_t$ and $\sum_{i=1}^n \vert \Delta q_{i,t}\vert = \vert q_t \vert$ 
with $n$ being the number of all possible boundary profiles that the convex combination between $x$ and $y$ intersects. We define $x_{k+1} =x_k +\Delta q_{k,t}\mathbf{1}_t, \ k \in \{1,...,n\},$ with $x_1 = x,$ and $ \ x_{n+1} =y $, such that each $x_{i+1} \ i \in \{1,...,n-1\}$ is a boundary profile, and therefore, $M(x_{i+1} + \epsilon\mathbf{1}_t)^Tx_{i+1} = M(x_{i+1}-\epsilon \mathbf{1}_t)^Tx_{i+1}$, holds from the definition of a boundary profile, where $\epsilon \in \mathbb{R} $ is sufficiently small.
%
\begin{small}
\begin{align*}
    &\abs{\Delta d_c} \\
    = &\abs{M_c(y)^{T}y - M_c(x)^{T}x}\\
    =& \abs{M_c(y)^{T}y -\sum_{i=2}^n\left( M_c(x_{i}+\epsilon\mathbf{1}_t)- M_c(x_{i}-\epsilon\mathbf{1}_t)\right)^{T}x_{i} - M_c(x)^Tx}\\
    =& \abs{M_c(y)^{T}y -\sum_{i=2}^n\left( M_c(x_{i}+\epsilon\mathbf{1}_t)- M_c(x_{i-1}+\epsilon\mathbf{1}_t)\right)^{T}x_{i} - M_c(x)^Tx}\\
    \leq & \abs{M_c(y)^{T}y - M_c(x_{n}+\epsilon\mathbf{1}_t)^{T}x_{n}} +\sum_{i=2}^{n-1} \abs{ M_c(x_{i}+\epsilon\mathbf{1}_t)^{T}(x_{i+1}-x_{i})}  \\ 
    &
    +\abs{M_c(x_{1}+\Delta\mathbf{1}_t)^{T}x_2 - M_c(x)^{T}x}\\
    \leq & \sum_{i=1}^n\abs{\Delta q_{i,t}} \mathbf{1} 
    = \abs{q_t}\mathbf{1} \\
    \implies & \abs{\Delta d_{c,i}} \leq \abs{q_t}, \ \forall i \in \{1,...,T\}
\end{align*}
\end{small}%
where the third equality holds, as incidence matrix $M_c$ remains same, i.e., $M_c(x_{i+1}-\epsilon \mathbf{1}_t) = M_c(x_{i} + (\delta q_i,t-\epsilon) \mathbf{1}_t) =  M_c(x_{i}+\epsilon \mathbf{1}_t)$, and the fourth inequality follows from the triangle inequality property.

Similarly, the second result holds as a consequence of
\begin{small}
\begin{align*}
    &\abs{\sum_{i=1}^T \Delta d_{c,i}} \\
    = & \abs{\mathbf{1}^T\Delta d_c}\\
    =& \abs{\mathbf{1}^T\left(M_c(y)^{T}y - M_c(x)^{T}x\right)}\\ 
    = & \abs{\mathbf{1}^T\big(M_c(y)^{T}y - \sum_{i=2}^n\left( M_c(x_{i}+\epsilon\mathbf{1}_t)- M_c(x_{i}-\epsilon\mathbf{1}_t)\right)^{T}x_{i} -M_c(x)^{T}x\big) }\\
    = & \abs{\mathbf{1}^T\big(M_c(y)^{T}y - \sum_{i=2}^n\left( M_c(x_{i}+\epsilon\mathbf{1}_t)- M_c(x_{i-1}+\epsilon\mathbf{1}_t)\right)^{T}x_{i} -M_c(x)^{T}x\big) }\\
    \leq & \abs{\mathbf{1}^T\left(M_c(y)^{T}y - M_c(x_{n}+\epsilon\mathbf{1}_t)^{T}x_{n}\right)} +\sum_{i=2}^{n-1} \bigg|\mathbf{1}^TM_c(x_{i}+\epsilon\mathbf{1}_t)^{T} \\
    & (x_{i+1}- x_{i})\bigg|+ 
    \abs{\mathbf{1}^T\left(M_c(x_{1}+\epsilon\mathbf{1}_t)^{T}x_2 - M_c(x)^Tx\right)}\\
    \leq & \sum_{i=1}^n\abs{\Delta q_{i,t}}
    \textrm{, using Lemma }\ref{lemma_3}\\
    = & \abs{q_t}
\end{align*}
\end{small}
This completes the proof.
\end{proof}

For any SoC profile $x$, the cost calculated from the cycle stress function using Rainflow cycle counting method, i.e., $C_s(x)$, is larger or equal than the cost calculated from cycle stress function using naive enumeration of the profile $x$ based on every switch between charging and discharging half-cycles, as summarized in the lemma below. Recall the set $S_r$ of switching time indices from steps~\ref{sgn_step_start}-\ref{sgn_step_end} in Algorithm~\ref{alg:rainflow}.

\begin{lemma}
\label{prop8}
Given a SoC profile $x$, using the step decomposition of the profile, i.e. $x = \sum_{t = 0}^{T} p_{t}\mathbf{1}_{t}$ as defined in Remark~\ref{step_decomposition}, the following holds,
\begin{align}
    C_s(x) \ge  
    \sum_{i=1}^{|S_r|-1} \Phi\left(\abs{\sum_{t=S_r[i]+1}^{S_r[i+1]}p_t}\right), 
\end{align}
where terms in the R.H.S represents consecutive half-cycles under naive enumeration of the SoC profile $x$. The set $S_r$ contains the time indices where the profile $x$ changes direction, steps~\ref{sgn_step_start}-\ref{sgn_step_end} in Algorithm~\ref{alg:rainflow}.  
\label{lemma_4}
\end{lemma}
\begin{proof}: WLOG, we assume that SoC profile $x$ has at most one cycle, and the proof is divided into sub cases.
\begin{itemize}
    \item Case 1: There is no cycle in the SoC profile $x$, i.e., set $S_r = \{0,T\}$. In this case equality holds, as the Rainflow algorithm simply returns the same vector of charging/discharging half-cycle depths as in the case of naive enumeration of profile $x$ based on every switch between charging and discharging half-cycles.

\item Case 2: There is a cycle, i.e., set $S_r = \{0,t_1,t_2,T\}$ with $0<t_1<t_2<T$ and $\Delta_1 = |x_0-x_{t_1}|=|p_1...+p_{t_1}|, \ \Delta_2 = |x_{t_1}-x_{t_2}|= |p_{t_1+1}+p_{t_1+2}...+p_{t_2}|$ and $\Delta_3 = |x_{t_2}-x_T| = |p_{t_2+1}+p_{t_2+2}...+p_{T}|$, such that $\Delta_1\geq \Delta_2,$ and $\Delta_3\geq \Delta_2$ hold. The cost of storage degradation from cycle stress function by naively enumerating profile $x$ is given by 
\begin{small}
\begin{align*}
    g(x) =& \sum_{i=1}^{|S_r|-1} \Phi\left(\abs{\sum_{t=S_r[i]+1}^{S_r[i+1]}p_t}\right)\\
    = &\left(\Phi\left(\abs{\sum_{t=1}^{t_1}p_t}\right)+\Phi\left(\abs{\sum_{t=t_1+1}^{t_2} p_t}\right)+\Phi\left(\abs{\sum_{t=t_2+1}^Tp_t}\right)\right) \\
    = & \phi(\Delta_1)+\phi(\Delta_2)+\phi(\Delta_3)
\end{align*}\end{small}%
However, using Rainflow algorithm, a full cycle of depth $\Delta_2$ is extracted and the depth vector $d = [\Delta_2, \Delta_2, |x_{T}-x_{0}|]^T = [\Delta_2, \Delta_2, \Delta_3+\Delta_2-\Delta_1]^T $. Therefore, the cost from cycle stress fucntion using Rainflow algorithm is as follows, 
\begin{small}
\begin{align*}
    C_s(x) = &\phi(\Delta_2)+\phi(\Delta_2)+\phi(\Delta_1+\Delta_3-\Delta_2)\\
    &\textrm{where, } \Delta_1+\Delta_3-\Delta_2 \ge 0, \ \text{using \cite[Proposition~4]{shi2017optimal}}\\
    \geq &\phi(\Delta_2)+\phi(\Delta_2)+\phi(\Delta_1)+\phi(\Delta_3) -\phi(\Delta_2) \\
    = &\phi(\Delta_2)+\phi(\Delta_1)+\phi(\Delta_3) = g(x)
\end{align*}\end{small}
\end{itemize}
The result could generalize for any number of cycles. To see this, notice that during extraction of full cycles as per Rainflow algorithm, for every cycle extracted, the intermediate cost of storage degradation increase as shown in the proof.   
\end{proof}
Now we provide the proof of Theorem 1 as below. Again, we only prove the result for charging half-cycle depths, which can generalize to discharging half-cycle depths. 

Consider an arbitrary SoC profile $x$ with T time steps, and let $y_{k}, ~k \in \{0,1,...,T\}$, represent the profile which is calculated by taking convex combination of $x$ and an arbitrary profile $q^k, \ k \in  \{0,1,...,T\}, \ q^{k} = \sum_{t = 0}^{T}q_t \mathbf{1}_t$, such that the SoC profile $q^{k}$ includes at most $k$ nonzero step changes or $k$ nonzero amplitudes,  
\begin{small}
\begin{align*}
    &y_{k} = \lambda x+(1-\lambda) q^{k}, \textrm{ and} \\
    &y_{k} = y_{k-1}+(1-\lambda) q^{1}
\end{align*}
\end{small}
In the proof, we evaluate the change in profile $x$ due to $q^{k}$ by evaluating the change in depth vector due to one step change perturbation at a time. 

Now, considering only charging half-cycle depths (it can be proved similarly for the discharging half-cycle depths), the Rainflow cycle algorithm gives charging half-cycle depths for $x$ and $y_1$ as,
\begin{small}
\begin{align*}
    &x : d_{c,1}(x), d_{c,2}(x),...,d_{c,m}(x),...,d_{c,M}(x),0,0...\\
    &y_{1} : d_{c,1}(y_1), d_{c,2}({y_1}),...,d_{c,n}({y_1}),...,d_{c,N}({y_1}),0,0,0...
\end{align*}
\end{small}%
where $d_{c,i}(x)$ and $d_{c,i}({y_1})$ represent the charging half-cycle depth obtained from $M_c(x)^Tx$ and $M_c(y_1)^Ty_1$, respectively. Further, define $\Delta d_{c,i}({y_1})$ to satisfy
\begin{small}
\begin{align*}
    d_{c,i}({y_1}) = \lambda d_{c,i}(x) +(1-\lambda)\Delta d_{c,i}({y_1}), \ \forall i= 1,2,...,T.
\end{align*}
\end{small}%
Re-iterating the process for k steps, we get 
\begin{small}
\begin{align*}
    d_{c,i}({y_k}) = &d_{c,i}({y_{k-1}})  +(1-\lambda)\Delta d_{c,i}({y_k})\\
    d_{c,i}({y_k})= &\lambda d_{c,i}(x) +(1-\lambda)\sum_{j=1}^k\Delta d_{c,i}({y_j})\\
    \implies d_{c,i}({y_k}) = & \lambda d_{c,i}(x) +(1-\lambda)\Delta d'_{c,i},\textrm{ }\forall i= 1,2,...,T\\
\textrm{s.t. } \begin{pmatrix}
    \Delta d'_{c,1}\\
    \Delta d'_{c,2}\\
    \vdots\\
    \Delta d'_{c,T}
  \end{pmatrix}:=& \begin{pmatrix}
    \Delta d_{c,1}({y_1}) & \Delta d_{c,1}({y_2}) & \dots & \Delta d_{c,1}({y_k}) \\
    \Delta d_{c,2}({y_1}) & \Delta d_{c,2}({y_2}) & \dots & \Delta d_{c,2}({y_k}) \\
    \vdots & \vdots & \ddots & \vdots \\
    \Delta d_{c,T}({y_1}) & \Delta d_{c,T}({y_2}) & \dots & \Delta d_{c,T}({y_k})
  \end{pmatrix} \begin{pmatrix}
    1\\
    1\\
    \vdots\\
    1
  \end{pmatrix} 
\end{align*}
\end{small}%
Then the change in depth vector $d_c$ can be written as
\begin{small}
\begin{align*}
   \sum_{i = 1}^{T}\Delta d'_{c,i} = &\sum_{i=1}^T\left(\sum_{j = 1}^{k}\Delta d_{c,i}({y_j})\right) =\sum_{j=1}^k\left(\sum_{i = 1}^{T}\Delta d_{c,i}({y_j})\right)\\
   \implies \sum_{i = 1}^{T}\Delta d'_{c,i} \leq &\abs{\sum_{i = 1}^{T}\Delta d'_{c,i}} 
   \leq \sum_{j=1}^k\abs{\sum_{i = 1}^{T}\Delta d_{c,i}({y_j})}   \leq \sum_{t=0}^T \abs{q_{t}}
\end{align*}
\end{small}%
where the second last term denotes the change in the depth vector due to each step change perturbation, and the last inequality holds from Proposition~\ref{proposition_1} for each step change perturbation. 

Assume for now $q_i$, $\forall i \in \{0,1,...,T\}$ is non-negative. The case where $q_i$ is arbitrary is discussed later. We can write the above inequality as
\begin{align*}
    &\sum_{i = 1}^{T}\Delta d'_{c,i} \leq \sum_{t=0}^T q_{t}
\end{align*}

The degradation cost for profile $y_k = \lambda x + (1-\lambda)q^k$, due to charging half-cycle depths  only (the result holds similarly for discharging half-cycle depths) is as follows, 
\begin{small}
\begin{align*}
    &C_s(\lambda x + (1-\lambda)q^k)\\
    =&\sum_{i = 1}^{T}\Phi\left(\lambda d_{c,i}(x) +(1-\lambda)\Delta d'_{c,i}\right)\\
    =&\sum_{i = 1}^{T+}\Phi\left(\lambda d_{c,i}(x) +(1-\lambda)\Delta d'_{c,i}\right)+ \sum_{i = 1}^{T-}\Phi\left(\lambda d_{c,i}(x) +(1-\lambda)\Delta d'_{c,i}\right)\\
    \intertext{{\normalsize where the set of $\Delta d'_{c,i}$ is divided into parts based on its sign. Using convexity of $\Phi$ and  \cite[Proposition~3]{shi2017optimal}, we have}}
&C_s(\lambda x + (1-\lambda)q^k)\\
    \leq& \sum_{i = 1}^{T}\lambda\Phi \left(d_{c,i}(x)\right) + \left(1-\lambda\right)\left\{\sum_{i = 1}^{T+}\Phi\left(\Delta d'_{c,i}\right)-\sum_{i = 1}^{T-}\Phi\left(\abs{\Delta d'_{c,i}}\right)\right\}
    \end{align*}
    \end{small}%
There are two sub cases. 
\begin{itemize}
    \item Case 1: $\sum_{i=1}^{T} \Delta d'_{c,i} = \sum_{t=0}^T q_{t}$, then,
\end{itemize}
    \begin{small}
    \begin{align*}
    &\sum_{i = 1}^{T}\lambda\Phi \left(d_{c,i}(x)\right) + \left(1-\lambda\right)\left\{\sum_{i = 1}^{T+}\Phi\left(\Delta d'_{c,i}\right)-\sum_{i = 1}^{T-}\Phi\left(\abs{\Delta d'_{c,i}}\right)\right\}\\
    \quad &\textrm{{\normalsize using \cite[Proposition~5]{shi2017optimal}}}\\
    &\leq \sum_{i = 1}^{T}\lambda\Phi \left(d_{c,i}(x)\right) +\left(1-\lambda\right) \Phi\left(\sum_{i = 1}^{T}\Delta d'_{c,i}\right)\\
    &=\sum_{i = 1}^{T}\lambda\Phi \left(d_{c,i}(x)\right) +\left(1-\lambda\right) \Phi\left(\sum_{t=0}^T q_{t}\right)\\
    &= \lambda C_s(x) +(1-\lambda)C_s\left(\sum_{t = 0}^{T}q_{t}\mathbf{1}_{t}\right) = \lambda C_s(x) +(1-\lambda)C_s(q^k)
    \end{align*}
    \end{small}%
    \begin{itemize}
        \item Case 2: $\sum_{i=1}^{T} \Delta d'_{c,i} < \sum_{t=0}^T q_{t}$, then,
    \end{itemize}
    \begin{small}
    \begin{align*}
    & \sum_{i = 1}^{T}\lambda\Phi \left(d_{c,i}(x)\right) + \left(1-\lambda\right)\left\{\sum_{i = 1}^{T+}\Phi\left(\Delta d'_{c,i}\right)-\sum_{i = 1}^{T-}\Phi\left(\abs{\Delta d'_{c,i}}\right)\right\}\\
    \leq & \sum_{i = 1}^{T}\lambda\Phi \left(d_{c,i}(x)\right) + \left(1-\lambda\right)\bigg\{\sum_{i = 1}^{T+}\Phi\left(\Delta d'_{c,i}\right)-\sum_{i = 1}^{T-}\Phi\left(\abs{\Delta d'_{c,i}}\right)\\
    &+ \sum_{i = T+1}^{T+m}\Phi\left(\abs{\Delta d'_{c,i}}\right)\bigg\}, \quad \textrm{ s.t.,}\sum_{i=1}^{T+m} \Delta d'_{c,i} = \sum_{t=0}^T q_{t}\\
    \leq & \sum_{i = 1}^{T}\lambda\Phi \left(d_{c,i}(x)\right) +\left(1-\lambda\right) \Phi\left(\sum_{i = 1}^{T+m}\Delta d'_{c,i}\right), \text{\cite[Proposition~5]{shi2017optimal}}\\
    =&\sum_{i = 1}^{T}\lambda\Phi \left(d_{c,i}(x)\right) +\left(1-\lambda\right) \Phi\left(\sum_{t=0}^T q_{t}\right)\\
    &= \lambda C_s(x) +\left(1-\lambda\right)C_s\left(\sum_{t = 0}^{T}q_{t}\mathbf{1}_{t}\right)= \lambda C_s(x) +(1-\lambda)C_s(q^k)
\end{align*}
\end{small}%
Hence, the cost is convex. Now, we verify this result when $q_{t}\in \mathbb{R}$, $\forall t \in\{0,1,...,T\}$. Let us assume $q_t < 0$ for $t = h$. Also, WLOG, we assume equality, as we can always use positive compensators if necessary, to make L.H.S = R.H.S, i.e.,
\begin{small}
\begin{align*}
    \sum_{i = 1}^{T}\Delta d'_{c,i} = \sum_{t=0,\ t \neq h}^T q_{t} + \abs{q_h} \, .
\end{align*}
\end{small}%
Now, the sum can be partitioned into three parts, by taking relevant combination of $\Delta d'_{c,i}$ (with positive compensators if necessary) such that there are three sets of $\Delta d'_{c,i}$, i.e., $\sum_{j=1}^{h^1}\Delta d'_{c,j}$ = $\sum_{t=0}^{h-1}q_t$, $\sum_{j=1}^{h^2}\Delta d'_{c,j}$=$\abs{q_{h}}$, and, $\sum_{j=1}^{h^3}\Delta d'_{c,j} $=$\sum_{t=h+1}^{T}q_t$. Essentially, the sum is partitioned in three parts based on switching between charging and discharging half-cycles, which is smaller or equal than the cost calculated from cycle stress function using Rainflow cycle counting method as shown in Lemma~\ref{lemma_4}.  
Then, it follows that
\begin{small}
\begin{align*}
    &C_s(\lambda x + (1-\lambda)q^k)\\
    \leq &\sum_{i = 1}^{T}\lambda\Phi \left(d_{c,i}(x)\right) + \left(1-\lambda\right)\left\{\sum_{i = 1}^{T+}\Phi\left(\Delta d'_{c,i}\right)-\sum_{i = 1}^{T-}\Phi\left(\abs{\Delta d'_{c,i}}\right)\right\}\\
     = & \sum_{i = 1}^{T}\lambda\Phi \left(d_{c,i}(x)\right) + \left(1-\lambda\right)\sum_{n=1}^3\left\{\sum_{j = 1}^{h^{n}+}\Phi\left(\Delta d'_{c,j}\right)-\sum_{j = 1}^{h^{n}-}\Phi\left(\abs{\Delta d'_{c,j}}\right)
    \right\}\\
    &\textrm{using \cite[Proposition~5]{shi2017optimal}},\\
    \leq& \sum_{i = 1}^{T}\lambda\Phi \left(d_{c,i}(x)\right) +\left(1-\lambda\right)\sum_{n=1}^3\left\{ \Phi\left(\sum_{j = 1}^{h^n}\Delta d'_{c,j}\right)
    \right\}\\
    =&\lambda C_s(x) +\left(1-\lambda\right) \left(C_s\left(\sum_{t=0}^{h-1}q_t\mathbf{1}_t\right)+C_s\left((q_h)\mathbf{1}_h\right)+C_s\left(\sum_{t=h+1}^Tq_t\mathbf{1}_t\right)\right)\\
    &\textrm{using \text{Lemma \ref{lemma_4}}}\\
    \leq & \left[\lambda C_s(x) +\left(1-\lambda\right)C_s\left(\sum_{t = 0}^{T}q_{t}\mathbf{1}_{t}\right)\right]= \lambda C_s(x) +(1-\lambda)C_s(q^k)
\end{align*}
\end{small}%
It is observed that this can be done similarly for any number of negative $q_i, \ i \in  \{0,1,..,T\}$. 

Note that the result holds for both charging half-cycle depths and discharging half-cycle depths independently. Given that the general degradation cost that accounts for both charging and discharging half-cycle depths is simply the sum of all individual depth degradation, the convexity still holds. 
In summary, the cost of degradation $C_s(x)$ is convex with respect to SoC profile $x$.

\end{section}
}{}

\bibliographystyle{IEEEtran}
\bibliography{sample}

\end{document}